\documentclass[11pt,letterpaper]{amsart}

\usepackage{amsmath,amsthm,amssymb,amsfonts}
\usepackage{mathrsfs}     % For script fonts
\usepackage{hyperref}     % For hyperlinks (check journal policy)
\usepackage{graphicx}     % For figures
\usepackage{enumerate}    % For customizable lists
\RequirePackage{rotating}

\usepackage{xcolor}
\usepackage{algorithm}
\usepackage{graphicx}
\usepackage{mathtools}
\usepackage{mathrsfs}
\usepackage{fix-cm}
\usepackage{multirow}
\usepackage{amsmath}
\usepackage{amssymb}
\usepackage{latexsym}
\usepackage{dsfont}
\usepackage{xcolor}
\usepackage{cite}
\usepackage{dsfont}
\usepackage{enumitem}
\usepackage{todonotes}
\usepackage{hyperref}
\usepackage{booktabs}
\usepackage{bbm}
\usepackage{framed}
\usepackage{mdframed}
\usepackage{pgfplots}
\pgfplotsset{compat = newest}
\usepackage{caption}
\usepackage{subcaption}
\usepackage{color}
\usepackage{graphicx,epsfig}
\usepackage{verbatim}
 
\usepackage{hyperref}

\usepackage{tabularray}
\UseTblrLibrary{siunitx}

\hypersetup{
colorlinks=true,
linkcolor=blue,
citecolor=darkgreen,
urlcolor=blue}
\definecolor{darkgreen}{RGB}{0,150,0}
\definecolor{C0}{RGB}{31,119,180}
\definecolor{C1}{RGB}{255, 127, 14}
\definecolor{MineShaft}{rgb}{0.188,0.188,0.188}
\definecolor{red}{RGB}{0,0,0}
\usepackage{algpseudocode}

\numberwithin{equation}{section}

\newcommand{\sgn}{{\rm sgn}\kern 0.12em}
\newcommand{\argmin}{{\rm argmin}\kern 0.12em}

\newcommand{\Ls}{\operatorname{Ls}}
\newcommand{\Li}{\operatorname{Li}}
\newcommand{\inner}[2]{\langle #1, #2 \rangle}

\newcommand{\R}{\mathbb{R}}
\newcommand{\N}{\mathbb{N}}
\newcommand{\bd}{{\rm bd}\kern 0.12em}

\newcommand{\gph}{{\rm gph}\kern 0.12em}
\newcommand{\dom}{{\rm dom}\kern 0.12em}
\newcommand{\Dom}{{\rm Dom}\kern 0.12cm}
\newcommand{\rge}{{\rm Rge}\kern 0.12em}
\newcommand{\epi}{{\rm epi}\kern 0.12em}
\newcommand{\hypo}{{\rm hypo}\kern 0.12em}
\newcommand{\inte}{{\rm int}\kern 0.12em}
\newcommand{\cl}{{\rm cl}\kern 0.12em}

\newcommand{\els}{\Gamma\text{-}\displaystyle\limsup_{n}}
\newcommand{\eli}{\Gamma\text{-}\displaystyle\liminf_{n}}
\newcommand{\weli}{ {\Gamma(w^\ast)\text{-}\displaystyle\liminf_{n}}}
\newcommand{\wels}{ {\Gamma(w^\ast)\text{-}\displaystyle\limsup_{n}}}

\newcommand{\elim}{\Gamma\text{-}\displaystyle\lim_{n}}
\newcommand{\welim}{ {\Gamma(w^\ast)\text{-}\displaystyle\lim_{n}}}
\newcommand{\wwelim}{ {\Gamma(w)\text{-}\displaystyle\lim_{n}}}

\usepackage{color}
\usepackage{colortbl}

\newcommand\norm[1]{\left\lVert#1\right\rVert}
\newcommand\dualnorm[1]{\left\lVert#1\right\rVert_{\ast}}
\newcommand\dual[1]{#1^{\ast}}
\newcommand{\adh}[1]{\overline{#1}}
\usepackage{tikz}
\def\oo{\sqsubset\!\!\!\sqsupset}
\newcommand{\rectangle}[1]{#1_{\oo}}
\providecommand{\abs}[1]{\lvert#1\rvert}

\let\epsilon\varepsilon
\let\subseteq\subset 
\theoremstyle{plain}
\newtheorem{theorem}{Theorem}[section]
\newtheorem{lemma}[theorem]{Lemma}
\newtheorem{proposition}[theorem]{Proposition}
\newtheorem{corollary}[theorem]{Corollary}

\theoremstyle{definition}
\newtheorem{definition}[theorem]{Definition}
\newtheorem{example}[theorem]{Example}

\theoremstyle{remark}
\newtheorem{remark}[theorem]{Remark}

\numberwithin{equation}{section}
\newcounter{claim}[theorem]

\title[Convergence of Convex Functions]{Duality for the $\Gamma$-Convergence of Convex Functions}

\author{Rafael Correa}
\address{Department of Mathematical Engineering and Center for Mathematical Modeling (CNRS IRL2807) , Universidad de Chile, Santiago, Chile}
\email{rcorrea@dim.uchile.cl}

\author{Pedro Pérez-Aros}
\address{Department of Mathematical Engineering and Center for Mathematical Modeling (CNRS IRL2807) , Universidad de Chile, Santiago, Chile}
\email{pperez@dim.uchile.cl}

\author{José Pablo Santander}
\address{Department of Mathematical Engineering, Universidad de Chile, Santiago, Chile}
\email{josesantander@ug.chile.cl}

\thanks{
Rafael Correa was partially supported by Centro de
			Modelamiento Matem\'{a}tico (CMM), ACE210010 and FB210005, BASAL funds for
			center of excellence and ANID-Chile grant: Fondecyt Regular 1240335,  P. P\'erez-Aros was partially supported by Centro de
			Modelamiento Matem\'{a}tico (CMM), ACE210010 and FB210005, BASAL funds for
			center of excellence and ANID-Chile grant: MATH-AMSUD 23-MATH-09 and
			MATH-AMSUD 23-MATH-17, ECOS-ANID ECOS320027, Fondecyt Regular 1220886,
			Fondecyt Regular 1240335, Fondecyt Regular 1240120 and Exploraci\'on  13220097. 
            José Pablo Santander was partially supported by Centro de
			Modelamiento Matem\'{a}tico (CMM), ACE210010 and FB210005, BASAL funds for
			center of excellence and ANID-Chile grant: Fondecyt Regular 1240335
}
\begin{document}

\begin{abstract}
We extend the duality principle for the $\Gamma$-convergence of convex lower semicontinuous functions, which was previously established only in separable reflexive Banach spaces, to the broader class of weakly compactly generated (WCG) Banach spaces, addressing a question of Fitzpatrick and Lewis.
Under the same classical hypothesis of equicoercivity, we show that $\Gamma$-convergence in the norm topology is equivalent to $\Gamma$-convergence of the Fenchel conjugates in the weak$^\ast$ topology.
We further prove that this duality is equivalent to the graphical convergence of the associated subdifferentials with respect to the product topology given by the norm on the primal space and the weak$^\ast$ topology on the dual.
The WCG setting encompasses all separable  and all reflexive Banach spaces separately, i.e, separable spaces without reflexivity assumptions and reflexive spaces without separability assumptions,
as well as important non-reflexive spaces which may fail to be separable,
such as $L^1(\mu)$ for an arbitrary $\sigma$-finite measure. 
As an application, we derive dual characterizations of the $\Gamma$-convergence of convex integral functionals on $L^p$ spaces ($1\leq p<\infty
$).

\end{abstract}

\maketitle

% ---------------------------------------------------------------
% PAPER BODY
% ---------------------------------------------------------------

\section{Introduction}

In optimization and the calculus of variations, one seeks a notion of convergence for a sequence of variational problems that ensures the convergence of both minimizers and minimal values. The first concept of this kind was introduced by Wijsman in finite dimensions under the name \emph{infimal convergence} \cite{Wijsman1964convexI,Wijsman1966convexII}, motivated by questions in statistical decision theory \cite{BurkholderWijsman1963}. The next step came from variational inequalities (see \cite{LionsStampacchia1967}), when Mosco introduced in the late 1960s a convergence notion in reflexive spaces for lower semicontinuous convex functions---and hence for closed convex sets via their indicator functions---in order to study the stability of solutions to variational inequalities \cite{mosco1967approximation,mosco1969convergence,Mosco1968}. This notion, now known as \emph{Mosco convergence}, was closely related to the convergence of solutions of elliptic and parabolic PDEs and the associated energy functionals (see, e.g., \cite{Spagnolo1968,DeGiorgiSpagnolo1973,Spagnolo1967cauchy}). In the 1970s, De Giorgi, Franzoni and the Italian mathematical school extended this perspective to general topological spaces and general functions under the name \emph{$\Gamma$-convergence} \cite{DeGiorgiFranzoni1975}, which has since become one of the central tools of the calculus of variations (see, for e.g, \cite{dalmaso1993introduction,gamma_convergence_for_beginners}).

The dualization of variational convergences for convex functions was initiated by Wijsman in finite dimension in \cite{Wijsman1966convexII}. It was continued and generalized by Mosco to the reflexive setting in  \cite{mosco1971continuity}, where he established that Mosco convergence renders the Fenchel conjugation bicontinuous. This phenomenon is, in fact, characteristic of reflexivity (see \cite{moscoreflexivity,beer1992slice}). Subsequently, Joly extended Mosco’s results to general Banach spaces (see \cite{joly}). For $\Gamma$-convergence itself, a duality was established in reflexive separable Banach spaces (see\cite{joly} and \cite{AmbrosettiSbordone1976GammaG}): 
\begin{theorem}\label{Theo:clasico}
Provided that
\begin{equation}\label{classical hyphotesis}
   \sup_{n \in \mathbb{N}} f_n(x_n) < \infty \implies \sup_{n \in \mathbb{N}} \|x_n\| < \infty.
\end{equation}

for every sequence in the separable reflexive space, then 

\begin{equation*}
    f = \wwelim f_n \iff \dual{f}=\Gamma(\norm{\cdot}_{\ast})\text{-}\lim_n \dual{f}_n
\end{equation*}

i.e, the  sequence of lower semicontinuous convex functions $(f_n)_{n\in\N}$ $\Gamma$-converges with the weak topology to a function $f$ if and only if the sequence of conjugates $\Gamma$-converges with the norm topology of the dual to the conjugate of $f$.
\end{theorem}
This classical result, although theoretical, had applications to convex integrals of the calculus of variations (see \cite{Marcellini1973ConvergenzaFunzioniConvesse, BoccardoMarcellini1976VariationalInequalitiesConvergence, MarcelliniSbordone1977DualitaPerturbazione} and references therein). On the other hand, McLinden and Bergstrom \cite{McLindenBergstrom1981} explored consequences in finite dimensions for convex sets, functions, and programming models.

A further milestone in the dualization of $\Gamma$-convergence was reached by Attouch in 1977, who proved that Mosco convergence of convex functions is equivalent to the graphical convergence of their subdifferentials with the product of the norm topologies
up to a normalization constant \cite{attouch1977}. This fundamental result, now known as \emph{Attouch theorem}, has become a cornerstone of variational analysis. It's applications are numerous: to stability in optimization and variational problems \cite{AzeAttouchWets1988,AubinWets1988Stable,AttouchWets1989EpigraphicalAnalysis,AttouchWets1991QuantitativeEpigraphical,AttouchWets1993NonlinearConditioning,AttouchWets1993III}, to the analysis of parabolic variational inequalities \cite{parabolic_variational_inequailities}, to numerical stability in optimization algorithms \cite{numerical_attouch}, and to structural results such as generalized second-order derivatives of convex functions \cite{rockafellar1990generalized, Rockafellar1988EpiDifferentiability, CominettiCorrea1990GenSecondOrder, Cominetti1991PseudoDifferentiability, LevyPoliquinThibault1995PartialExtensions, Poliquin1992} and differentiability of Lipschitz set-valued mappings \cite{daniilidis2024extending}. Attouch’s theorem also remains influential in applied areas such as nonsmooth mechanics and the study of subgradient evolution systems \cite{attouch2014variational}.

\section{The contribution}

Despite these developments, the duality theory of $\Gamma$-convergence (Theorem \ref{Theo:clasico}), has classically been
confined, to the best of our knowledge, to separable reflexive Banach spaces, a restriction which
arises from genuine intrinsic technical obstacles in the available approaches, but leaves out separable spaces that are not reflexive such as $C(K)$ when K is a compact metric space, as well as reflexive spaces that may fail to be separable and even spaces that are not reflexive that also may fail to be separable, such as $L^2(\mu)$ for an arbitrary measure $\mu$ and $L^1(\mu)$ when $\mu$ is $\sigma$-finite, respectively. It is therefore natural to ask whether these assumptions are inherent in the duality mechanism itself. 

In this work, relying on more recent advances in Banach space theory, we show that
these constraints can be removed: separability and reflexivity are not required. Even in
the reflexive setting separability is not necessary, and the duality framework can in fact be extended
to the broader class of weakly compactly generated spaces. Within this framework, we establish a Fenchel duality principle for
$\Gamma$-convergence in weakly compactly generated Banach spaces between the norm topology of the primal and the weak* topology of the dual, where the novelty is that the main result of this article is valid for every reflexive space (without separablity assumptions), for every separable space (without reflexivity assumptions), and also for important non reflexive spaces that may fail to be separable, such as $L^1(\mu)$ whenever $\mu$ is a $\sigma$-finite measure. In particular, this extends the classical duality results previously
known in reflexive separable spaces (Theorem \ref{Theo:clasico}) and recovers the reflexive case as a
special instance (without the necessity of separability).

Our results provide an answer to a question raised explicitly by Fitzpatrick and Lewis at the end of their work in 1988 \cite{FitzpatrickLewis2006WeakStarConvexSets},
concerning the relation between the convergence of convex functions and
the convergence of their Fenchel conjugates. This question of Fitzpatrick and Lewis was raised at a time where continuity of the Fenchel conjugate and the polar map, and generalizations of Mosco convergence to non-reflexive spaces, as well as topologies for closed (or closed and convex) sets and convex functions, were extensively studied (see, e.g., \cite{Back1986continuity,Beer1988YoungFenchel,Beer1990EpiDistance,AttouchWets1983SaddleConvergence,DoleckiSalinettiWets1983,AttouchWets1986Isometries,ContessePenot1991,Penot1991CosmicBoundedHausdorff,BeerThera1994AWDifferentialOperator, SonntagZalinescu1993SetConvergences, BeerLucchetti1993WeakTopologiesClosedSubsets, BeerLucchetti1991ConvexOptimizationEpiDistance, BorweinFitzpatrick1989, BorweinVanderwerff1995LipschitzRegularizations, BeerDiConcilio1991UniformContinuity, Beer1994WijsmanRenorming, Penot1992TopologiesConvergences}).
 
We further use this duality framework to establish a relation between $\Gamma$-convergence of convex functions and the graphical convergence of the associated subdifferentials with respect to the norm–$w^\ast$ product topology (see Theorem \ref{Theo:equival}). We provide a counterexample that demonstrate that these relations break down outside the weakly compactly generated setting (see Example \ref{Ex:necesidad_WCG}). In the reflexive case, and under equicoerciveness, we additionally prove that Mosco convergence—equivalently, by Attouch’s theorem, graphical convergence in the norm–norm product topology—is equivalent to graphical convergence simultaneously in the norm–weak$^\ast$ and weak–norm product topologies. Moreover, under equicoerciveness of the functions or equicoerciveness of their conjugates, we show that Mosco convergence in reflexive spaces is equivalent to the $\Gamma$-convergence of the functions with the norm topology and of their Fenchel conjugates with the norm topology of the dual when the $\Gamma$-limits are conjugate to each other, which, under equicoerciveness of the functions and their conjugates, is equivalent to the $\Gamma$-convergence of the functions with the weak topology and of their Fenchel conjugates with the weak* topology when the $\Gamma$-limits are conjugate to each other.

As an application, we consider convex integral functionals on $L^p$ spaces
over $\sigma$-finite measure spaces. This framework goes beyond the separable reflexive classical
setting, since such spaces may fail to be separable, and in the case
$p=1$ they are not reflexive. Nevertheless, they are weakly compactly generated,
so the main results of this work apply and yield dual characterizations of the corresponding
variational convergence.

The proof of our main result relies heavily on modern tools of Banach Space Theory. Our approach also reveals insights on the structural properties of weakly compactly generated Banach spaces, analogous to those used in the recent characterization of Asplund spaces by Fabian and Cúth in \cite{rich_families_asplund}.

\section{Outline of the Manuscript}

This section provides a brief overview of the structure of the manuscript. 
Section~\ref{Section_Preliminaries} introduces the notation and mathematical preliminaries. 
Section~\ref{MainResults} presents the main results concerning limits of functions and subdifferentials, together with illustrative examples. 
In particular, Theorem~\ref{previo_a_attouch} shows that even when the $\Gamma$-limit of a sequence of functions does not exist, it is still possible to relate the lower sequential  $\Gamma$-limit of the functions with the upper sequential $\Gamma$-limit of their conjugates in the weak$^\ast$-topology. 
Although this may appear to be only an intermediate step, it represents the most technically demanding part of the paper. 
Corollary~\ref{dualidad} proves that, provided one of the limits exists, the $\Gamma$-limit of the functions coincides with the $\Gamma$-limit of their conjugates in the weak$^\ast$-topology. 
Theorem~\ref{Theo:equival}, which is the main result of the paper, establishes that these conditions are also equivalent to the (sequential) graphical convergence of the subdifferentials with respect to the strong–$w^\ast$ product topology. Finally, Section~\ref{application} shows an application of the results of Section~\ref{MainResults} to integral functionals given by convex integrands.

The remaining sections are devoted to the proofs. Section~\ref{Some_Pre_results} gathers several auxiliary results needed in the sequel. Section~\ref{Section:Proof_previo_a_attouch} develops the proof of Theorem~\ref{previo_a_attouch}, the most substantial argument of this work. The proof is first established in separable Banach spaces (Subsection~\ref{Proof_previo_a_attouch_Separable}) and then extended to the general case by means of separable reduction. The method relies on rich families and projectional skeletons: Subsection~\ref{Sec_Sep_red_rich} introduces the necessary concepts, while Subsection~\ref{section_ResultsOn_SEPRED} develops the construction of rich families that remove the separability assumption. Although it is used for technical purposes in this paper, it should be noted that Corollary \ref{surjective_isometry} is of independent interest, as it goes in line with the characterization of Asplund spaces given in \cite{rich_families_asplund}. Finally, Subsection~\ref{ProofMain_WCG} completes the argument in weakly compactly generated Banach spaces. Section~\ref{Section:Proof_dualidad} presents the short proof of Corollary~\ref{dualidad}, and Section~\ref{Section:Proof_Theo:equival} concludes with the proof of Theorem~\ref{Theo:equival}.

\section{Mathematical Preliminaries}\label{Section_Preliminaries}

\subsection{Notations and basic definitions}

In what follows, $X$ will denote a Banach space with norm $\norm{\cdot}$. Its dual space will be denoted by $\dual{X}$, and $\norm{\cdot}_\ast$ stands for the norm in $\dual{X}$. The open ball of radius $\epsilon>0$ centered at $\bar x$ is denoted by $\mathbb{B}_\epsilon(\bar x)$. For a subset $A\subseteq X$, $\operatorname{\overline{sp}}(A)$ will denote the closed linear span generated by $A$.  

We denote by $w$ the weak topology of $X$, and by $w^\ast$ the weak$^\ast$ topology of $\dual{X}$. If a sequence $(x_n)_{n\in\N}$ in $X$ converges strongly to $x \in X$, we write $x_n \to x$. Weak convergence in $X$ is denoted by $x_n \rightharpoonup x$. Similarly, if $(x_n^\ast)_{n\in\N} \subseteq \dual{X}$ converges strongly to $x^\ast \in \dual{X}$, we write $x_n^\ast \to x^\ast$, whereas $x_n^\ast \overset{\ast}{\rightharpoonup} x^\ast$ denotes weak$^\ast$ convergence.  

The Mackey topology $\dual{\mu} := \mu(\dual{X},X)$,  defined on $X^\ast$, is the topology of uniform convergence on the family of all convex, symmetric, and weakly compact subsets of $X$. For a net $(x_\alpha^\ast) \subseteq \dual{X}$, we write $x_\alpha^\ast \overset{\mu^\ast}{\to} x^\ast$ to denote convergence to $x^\ast$ in the Mackey topology.  

The Banach space $X$ is said to be a \emph{weakly compactly generated} (WCG) space if there exists a weakly compact set $A \subseteq X$ such that the linear span of $A$ is dense in $X$ with respect to the norm topology. In this definition, $A$ may be assumed, without loss of generality, to be convex and symmetric (see, e.g., \cite[Chapter~13]{Libro_funcional}). We recall that every separable or reflexive space is WCG. There are also examples of WCG spaces that are not reflexive and not necessarily separable, such as $L^1(\mu)$ for any $\sigma$-finite measure $\mu$, as well as the $C(K)$ space whenever $K$ is a compact topological space homeomorphic to a weakly compact subset of a Banach space endowed with the restriction of the weak topology (see \cite[Chapter~13 and Theorem 14.9]{Libro_funcional}).

For a function $f\colon X \to \mathbb{R} \cup \{\infty\}$, we define its (effective) \emph{domain} and \emph{epigraph} as follows:
\[
\mathrm{dom}\, f := \{\, x \in X : f(x) < \infty \,\}, 
\qquad 
\mathrm{epi}\, f := \{\, (x, r) \in X \times \mathbb{R} : f(x) \leq r \,\}.
\]

Furthermore, a function  $f\colon X\to \mathbb{R} \cup \{\infty\}$ will be called \emph{proper} if $\dom f\neq\emptyset$, it will be called \emph{convex} if $\epi f$ is a convex set, and will be called \emph{lower semicontinuous} if  $\epi f$ is closed . We can also define the

  For a   function $f\colon X\to \R\cup\{-\infty, \infty\}$ 
    its \emph{Fenchel conjugate} is defined by 
    \begin{equation*} 
    f^{\ast}(x^{\ast})  := \displaystyle\sup_{x\in X}\left\{\inner{x^{\ast}}{x} - f(x)\right\}.
    \end{equation*}
    Furthermore, for a function $h: X^\ast \to \R\cup \{ -\infty,  \infty\}$, its conjugate will be understood defined on $X$ as
    \begin{align*}
        h^\ast (x) : = \displaystyle\sup_{x^\ast\in X}\left\{\inner{x^{\ast}}{x} - h(x^\ast)\right\}.
    \end{align*}

    For a given set $C\subseteq X$ we denote by  $\delta_C$  the \emph{indicator} function of $C$, that is, the function from $X$ to $\R \cup \{\infty\}$ defined by   
      $$\delta_C (x)= \left\{ \begin{array}{lcc}
             0, &  \text{ if } & x\in C,\\
             \\ \infty, & \text{ if }& x\notin C. \\
             \end{array}
   \right.$$

The \emph{support} function of $C\subseteq X$, denoted by $\sigma_C$, is the function from $\dual{X} \to \R\cup\{\infty\}$ such that $\sigma_C(\dual{x}):= \sup\{\inner{\dual{x}}{x}: x\in C\}$. Observe that $\delta_C^{\ast}= \sigma_C$.

     Given a proper lower semicontinuous convex function $f\colon X \to \R\cup\{\infty\}$ and $\lambda>0$, we recall that the Moreau envelope of $f$ associated to $\lambda$, denoted by $f_{\lambda}$, is the following map:
    \begin{align*}
        f_{\lambda}(x) := \displaystyle\inf_{y\in X}\left\{f(y) + \frac{1}{2\lambda}\norm{x-y}^2\right\}.
    \end{align*}
    More generally, for two functions $f,g\colon X\to\R\cup\{\infty\}$, the inf-convolution is denoted by \begin{align*}
      f  \square g(x):= \inf\left\{f(x_1)+g(x_2) : x_1+ x_2 =x\right\}.
    \end{align*}

     For lsc proper convex function $f\colon X \to \mathbb{R} \cup \{\infty\}$ one denotes by $\partial f$ the graph of the subdifferential of $f$:
\[
\partial f := \{(x, x^*) \in X \times X^* : x^* \in \partial f(x)\},
\]
where $\partial f(x)$ is the subdifferential of $f$ at $x$ and is given by:
\[
\partial f(x) := \{x^* \in X^* : f(x) + \dual{f}(\dual{x}) = \inner{\dual{x}}{x}\}.
\]
We will denote by $\partial \dual{f}$ the graph of the subdifferential of $\dual{f}$ restricted to $\dual{X}\times X$ rather than $\dual{X}\times X^{**}$. So $x\in\partial \dual{f}(\dual{x})$ if and only if $\dual{x}\in\partial f(x)$.\\

\subsection{Limits of sequences of functions and operators}
 Let $(Y, \tau)$ be a topological space, and let $(f_n)_{n\in\N}$ be a sequence of functions from $Y$ into $\mathbb{R} \cup \{\infty\}$.  
We define two limit functions, called the \emph{lower sequential $\Gamma$-limit} and the \emph{upper sequential $\Gamma$-limit} of the sequence $(f_n)_{n\in\N}$ (with respect to $\tau$), as follows:
\begin{align}\label{Def_epi_limits}
\begin{aligned}
    \Gamma(\tau)\text{-}\liminf_{n} f_n(x) &:= \inf \left\{ \liminf_{n} f_n(x_n) : x_n \overset{\tau}{\to} x \right\}, \\
   \Gamma(\tau)\text{-}\limsup_{n} f_n(x) &:= \inf \left\{ \limsup_{n} f_n(x_n) : x_n \overset{\tau}{\to} x \right\},
\end{aligned}
\end{align}
where $x_n \overset{\tau}{\to} x$ denotes convergence of the sequence $(x_n)_{n\in\N}$ with respect to the topology $\tau$. Moreover, it is a known fact that when the topology $\tau$ is metrizable (or more general, first countable), then each infimum is indeed a minimum \cite[p.30, Remark 1.26]{gamma_convergence_for_beginners}.
%It is well known that if $Y$ is first-countable, then both infima are actually minima (see \cite{Libro_Attouch}).

Furthermore, we say that the sequence $(f_n)_{n\in\N}$ \emph{sequentially $\Gamma$-converges } under the topology $\tau$ if
\begin{equation*}
\Gamma(\tau)\text{-}\liminf_{n} f_n = \Gamma(\tau)\text{-}\limsup_{n} f_n.
\end{equation*}

In this case, we denote its sequential $\Gamma$-limit simply by $\Gamma(\tau)\text{-}\displaystyle\lim_{n} f_n$.

Since we are particularly interested in $\Gamma$-limits with respect to the norm topology, in this case, we omit the symbol with respect to the topology  for the sequential $\Gamma$-limits of sequences $f_n \colon X \to \mathbb{R} \cup \{\infty\}$, $n\in\N$ (or defined in $\dual{X}$ instead of $X$). That is: 
\begin{equation*}
	\begin{aligned}
		\eli f_n &:= \Gamma(\norm{\cdot})\text{-}\liminf_n f_n,\\  \els f_n &:= \Gamma(\norm{\cdot})\text{-}\limsup_n f_n,\\ \elim f_n &:= \Gamma(\norm{\cdot})\text{-}\lim_n f_n.
	\end{aligned}
\end{equation*}

Let $(S_n)_{n\in \mathbb{N}}$ be a sequence of sets in a topological space $(Y,\tau)$.  
The \emph{$\tau$-sequential lower limit}, denoted by $\Li_{\tau} S_n$, is the set of all $y \in Y$ for which there exists a sequence $(y_n)_{n\in \mathbb{N}} \subseteq Y$ with $y_n \in S_n$ for all sufficiently large $n \in \mathbb{N}$ and $y_n \xrightarrow{\tau} y$.   Similarly, the \emph{$\tau$-sequential upper limit}, denoted by $\Ls_{\tau} S_n$, is the set of all $y \in Y$ for which there exist an increasing function $j\colon\mathbb{N}\to\mathbb{N}$ and a sequence $(y_n)_{n\in \mathbb{N}}$ with $y_n \in S_{j(n)}$ such that $y_n \xrightarrow{\tau} y$.

We now introduce the notion of \emph{graphical convergence} for set-valued mappings by identifying them with their graphs. Recall that for a set-valued mapping $A\colon Y \to Z$, its graph is defined as
\[
    \gph A := \{ (y,z) \in Y \times Z : z \in A(y) \}.
\]
Given a sequence of set-valued mappings $(A_n)_{n \in \mathbb{N}}$ from $Y$ to $Z$, the \emph{$\tau_Y \times \tau_Z$-sequential lower limit} and the \emph{$\tau_Y \times \tau_Z$-sequential upper limit} are defined as the set-valued mappings
\[
    \Li_{\tau_Y \times \tau_Z} A_n
    \quad \text{and} \quad
    \Ls_{\tau_Y \times \tau_Z} A_n,
\]
whose graphs are the sets
\[
    \Li_{\tau_Y \times \tau_Z} \gph A_n
    \quad \text{and} \quad
    \Ls_{\tau_Y \times \tau_Z} \gph A_n,
\]
respectively.

Moreover, we say that a sequence of set-valued mappings $A_n \colon Y \rightrightarrows Z$, $n\in\N$, 
\emph{sequentially $\tau_Y \times \tau_Z$-converges} to a set-valued mapping 
$A \colon Y \rightrightarrows Z$, and we write
\begin{equation}\label{def_graph_conv}
    A_n \xrightarrow{ G(\tau_Y \times \tau_Z)} A,
\end{equation}
if
\[
    \Ls_{\tau_Y \times \tau_Z} A_n \subseteq \gph A \subseteq \Li_{\tau_Y \times \tau_Z} A_n,
\]
which is the same as  \[
    \Ls_{\tau_Y \times \tau_Z} A_n = \gph A = \Li_{\tau_Y \times \tau_Z} A_n.
\]

\section{Limits of convex functions and subdifferentials}\label{MainResults}

In this section we study the relation between the $\Gamma$-convergence of a sequence of convex functions and the $w^\ast$-sequential limit of its conjugate functions, in the sense of \eqref{Def_epi_limits}.

The first result of this section provides a characterization of the conjugate of the lower sequential  $\Gamma$limit of a sequence of functions,  as the upper sequential $\Gamma$-limit of their corresponding conjugates with respect to the weak$^\ast$-topology. The proof of this result will be presented in Section \ref{Section:Proof_previo_a_attouch}.

\begin{theorem}\label{previo_a_attouch}
	Let $X$ be a   WCG Banach space, and let $f_n\colon X \to \R \cup \{ \infty\}$, $n\in\N$, be a sequence of lower semicontinuous proper convex functions such that the sequence $(f_n^*)_{n\in\N}$ satisfies the  equicoercivity property 
	\begin{align}\label{equicoercividad}
		 (x_n^*)_{n\in\N} \subseteq X^*, \quad \sup_{n \in \mathbb{N}} f_n^*(x_n^*) < \infty \implies \sup_{n \in \mathbb{N}} \|x_n^*\|_{\ast} < \infty.
	\end{align}
 Then it holds that $$(\eli f_n)^* = \wels f_n^*,$$ provided that    $\dom \els f_{n}\neq \emptyset$. 
\end{theorem}
 {\color{red}
 The following example emphasizes
that the condition $\dom \els f_n \neq \emptyset$ is necessary for the validity of Theorem~\ref{previo_a_attouch}.

\begin{example}
	Consider the sequence of functions $f_n(x) := \tfrac{1}{n}|x| + n$.  
	On one hand, we have
	\begin{align*}
		f(x) := \eli f_n(x) = \infty,
	\end{align*} 
	which implies $f^\ast \equiv -\infty$.  
	On the other hand, 
	\begin{align*}
		f_n^\ast(x^\ast) = \delta_{[-\tfrac{1}{n}, \tfrac{1}{n}]}(x^\ast) - n,
	\end{align*}
	and therefore
	\begin{equation*}
		\wels f_n^\ast(x^\ast) = 
		\begin{cases}
			- \infty & \text{if } x^\ast = 0, \\[0.3em]
			 \infty & \text{if } x^\ast \neq 0.
		\end{cases}
	\end{equation*}
	Hence, $f^\ast \neq \wels f_n^\ast$.
\end{example}

 }
 
In the following example, we show that, in the last corollary, the assumption that $X$ is WCG cannot, in general, be removed for an arbitrary Banach space $X$.

\begin{example}\label{Ex:necesidad_WCG}
Let $X$ be a Banach space, and suppose there exists a sequence of uniformly bounded, $w^\ast$-closed, convex sets $C_n \subset X^\ast$ such that $C_n \subset C_{n+1}$ for all $n \in \mathbb{N}$, and  
\[
 \Ls_{w^\ast} C_n \neq C := \cl^{w^\ast}\!\Big( \bigcup_{n \in \mathbb{N}} C_n \Big).
\]
Define functions $f, f_n : X \to \mathbb{R}$, $n\in\N$, by
\[
f_n(x) := \sigma_{C_n}(x), 
\qquad 
f(x) := \sigma_{C}(x).
\]
It is straightforward to verify that $\eli f_n = f$, which in particular implies
\[
\dom(\els f_n) \neq \emptyset.
\]

Moreover, since $f_n^\ast = \delta_{C_n}$, the uniform boundedness of the sets $C_n$ immediately yields the equicoercivity property \eqref{equicoercividad}.  

Now, let $x^\ast \in C \setminus \Ls_{w^\ast} C_n$. On one hand, we have $f^\ast(x^\ast) = 0$. On the other hand, suppose that there exists a sequence $x_n^\ast \to x^\ast$ with $\displaystyle\limsup_n \dual{f}_n(\dual{x}_n) < \infty$. Then the inclusion $x_n^\ast \in C_{n}$ must hold for all sufficiently large $n \in \mathbb{N}$; otherwise, we would obtain $\displaystyle\limsup_n \dual{f}_n(\dual{x}_n) = \infty$. However, this would imply $x^\ast \in  \Ls_{w^\ast} C_n$, which is a contradiction. Therefore, the only possible alternative is  
\[
\limsup_{n} f_n^\ast(x_n^\ast) = \infty.
\]
By the arbitrariness of the sequence $x_n^\ast \to x^\ast$, we conclude that $\wels f_n(x^\ast)=\infty$. Consequently, Theorem~\ref{previo_a_attouch} does not hold in this setting.

\medskip

To illustrate this with a concrete example, let us consider $X=\ell_\infty(\N)$, and define the sets
\begin{equation*}
C_k := \left\{  (x_i)_{ i \in \N } \in \ell_1(\N)  : x_j =0 \text{ for all } j \geq k,\ \sum_{ j\in \N} | x_j| \leq 1 \right\}. 
\end{equation*}
These sets can be regarded as subsets of $X^\ast$, using the identification $\ell_1(\N) \subset X^\ast = \dual{\ell_\infty(\mathbb{N})} = \left(\ell_1(\N)\right)^{\ast\ast}$. It is not difficult to show that each $C_k$ is $w^\ast$-closed and convex. Furthermore, by Goldstine’s theorem (see, e.g., \cite[Theorem 3.96]{Libro_funcional}), we obtain
\begin{align*}
    C := \cl^{w^\ast}\left( \bigcup_{n \in \mathbb{N}} C_n \right) 
    =  \mathbb{B}_{X^\ast} :=\{ x^\ast \in X^\ast : \| x^\ast \|_{\ast} \leq 1  \}.
\end{align*}

Finally, by the Odell--Rosenthal theorem (see, e.g., \cite{MR377482} or \cite[Chapter XIII, Theorem 10]{MR737004}), the closed unit ball in $\ell_1(\N)$ is not weak$^\ast$ sequentially dense in the closed unit ball of $X^\ast$. Particularly, it implies that 
\[
\Ls_{w^\ast} C_n \neq C.
\]

	\end{example}

\begin{remark}\label{Remark_equicoerc}
It is worth noting that the equicoercivity property given in \eqref{equicoercividad} is true when assuming that for each $t\in\R$ there is a $\dual{w}$-compact set $K_t\subset\dual{X}$ such that for every $n\in\N$, $\{\dual{x}\in\dual{X}: \dual{f}_n(\dual{x})\leq t\}\subseteq K_t$, which is a standard hypothesis to assure convergence of minima in $\Gamma$-convergence. Also, it can be readily verified under the stronger assumption of \emph{uniform coercivity}, that is,  there exist constants $\alpha > 0$ and $\beta \in \mathbb{R}$ such that  
\begin{align}\label{Unif:equicoer}
    f_n^\ast(x^\ast) \geq \alpha \|x^\ast\| + \beta,
    \quad \text{for all } x^\ast \in X^\ast \text{ and all } n \in \mathbb{N}.
\end{align}
By taking the Fenchel conjugate in \eqref{Unif:equicoer} (see, e.g., \cite[Fact~4.4.9]{Libro_borwein}), one obtains that this condition is equivalent to the uniform boundedness of the functions in a neighborhood of the origin, which means, there exist $\rho, \varepsilon > 0$ such that  
\begin{align}\label{Unf_boundeness}
    f_n(x) \leq \rho,
    \quad \text{for all } x \in \mathbb{B}_\varepsilon(0) \text{ and all } n \in \mathbb{N}.
\end{align}
Furthermore, by \cite[Theorem~2.2.22]{zalinescu_convex_analysis}, this is in turn equivalent to the existence of a convex open neighborhood $U$ of the origin such that, for every $x \in U$, the sequence $(f_n(x))_{n \in \mathbb{N}}$ is bounded. Moreover, in this last case, i.e, when condition \eqref{Unf_boundeness} is satisfied, the $\Gamma$-limit of the sequence $(f_n)_{n\in\N}$ coincides with its pointwise limit in $\mathbb{B}_\varepsilon(0)$ (see \cite[Proposition~5.12]{dalmaso1993introduction}).
\end{remark}

Now, let us recall the following result provided in \cite{AubinFrankowska1990}[Theorem~7.5.1] and attributed to Mosco (see also \cite[Theorem~2]{penot2004continuity}). 

\begin{theorem} \label{teorema_2} 
	Let $X$ be a Banach space and let $f_n\colon X \to \R \cup \{ \infty\}$, $n\in\N$, be a sequence of lower semicontinuous proper convex functions.  
	Suppose there exists a bounded sequence $(x_n^\ast)_{n\in\N}$ in $X^\ast$ such that the sequence $(f_n^\ast(x_n^\ast))_{n\in\N}$ is bounded above.  
	Then the following equality holds:
	\begin{align*}
		(\weli f_n^\ast)^* = \els f_n.
	\end{align*}
\end{theorem}

As an immediate consequence of the two preceding results, we obtain a dual formula that relates the $\Gamma$-limit in the strong topology to the $\Gamma$-limit of the corresponding conjugates with respect to the weak$^\ast$ topology. The proof of this result is presented in Section \ref{Section:Proof_dualidad}.

\begin{corollary}\label{dualidad}
	Let $X$ be a Banach space and let  
	$f, f_n\colon X \to \R \cup \{ \infty\}$, $n\in\N$, be   lower semicontinuous proper convex functions. Then, 
    \begin{align*}
    f^{\ast} = \welim f_n^{\ast} \Rightarrow f = \elim f_n.
    \end{align*}
    Additionally, assume  that $X$  is a WCG Banach space and the sequence $(f_n^{\ast})_{n\in\N}$ satisfies the equicoercivity property given in \eqref{equicoercividad}. Then 
	\begin{align*}
		f = \elim f_n  
		\quad \Longleftrightarrow \quad 
		f^{\ast} = \welim f_n^{\ast}.
	\end{align*}
\end{corollary}

Next, we establish a connection between the $\Gamma$-limit of a sequence of convex functions and the graphical convergence of the subdifferentials (see \eqref{def_graph_conv}). It is worth noting that the $\|\cdot\|\times w^\ast$ sequential convergence of subdifferentials has already been investigated in \cite{Zagrodny2005WeakStarConvergenceSubdifferentials}. More precisely, in \cite[Theorem 3.1]{Zagrodny2005WeakStarConvergenceSubdifferentials} the authors proved that if a sequence of proper, lower semicontinuous convex functions Mosco-converges to $f$ and satisfies \eqref{Unf_boundeness}, then the corresponding subdifferentials converge with respect to $\|\cdot\|\times w^\ast$ in the sense of \eqref{def_graph_conv}.

In what follows, we show that this implication can be strengthened and is sharper. 
The proof of this equivalence is deferred to Section~\ref{Section:Proof_Theo:equival}.

\begin{theorem}\label{Theo:equival}
	Let $X$ be a WCG space, and let $f, f_n\colon X\colon \to \mathbb{R} \cup \{\infty\}$, $n \in \mathbb{N}$, be proper lower semicontinuous convex functions. 
Consider the following statements:
\begin{enumerate}[label=(\alph*)]
    \item $\dual{f} = \welim \dual{f}_n$.
    \item $\partial f_n\xrightarrow{G(\|\cdot\| \times \dual{w})} \partial f$ 
    and there exists $(a,a^*) \in \partial f$ and a sequence $(a_n,a_n^*) \in \partial f_n $ such that
    $a_n \to a,\; 
  a_n^* \overset{\ast}{\rightharpoonup} a^* \text{ and }  
    f_n(a_n) \to f(a).$
    \item $f = \elim f_n$.
\end{enumerate}
Then, \((a) \Rightarrow (b) \Rightarrow (c)\). Moreover, if the sequence $(\dual{f}_n)_{n\in\N}$ satisfies the equicoercivity condition \eqref{equicoercividad}, 
then the three statements are equivalent.
	
\end{theorem}

  Observe that \cite[Example 3.2]{Zagrodny2005WeakStarConvergenceSubdifferentials} also serves as a counterexample to Theorem \ref{Theo:equival}, showing that $(c)$ does not imply $(b)$ once the WCG hypothesis is dropped.

  In the case of a locally equibounded sequence of functions on a normed space, $\Gamma$-convergence coincides with pointwise convergence; hence, we deduce the following corollary.

  \begin{corollary}\label{Cor:pontwiseconvergence}
      Let $X$ be a WCG space, and let $f, f_n\colon X\colon \to \mathbb{R} \cup \{\infty\}$, $n \in \mathbb{N}$, be proper lower semicontinuous convex functions. In addition, suppose that $(f_n)_{n\in\N}$ is locally equibounded, i.e, for every $x\in X$, there exists $r >0$ and $M \in\R$ such that 
      \begin{equation*}
          \forall n\in\N, \forall y\in \mathbb{B}_{r}(x),\; |f_n(y)|\leq M.
      \end{equation*}
Then, the following statements are equivalent:
\begin{enumerate}[label=(\alph*)]
    \item $\dual{f} = \welim \dual{f}_n$.
    \item $\partial f_n\xrightarrow{G(\|\cdot\| \times \dual{w})} \partial f$ 
    and there exists $(a,a^*) \in \partial f$ and a sequence $(a_n,a_n^*) \in \partial f_n $ such that
    $a_n \to a,\; 
  a_n^* \overset{\ast}{\rightharpoonup} a^* \text{ and }  
    f_n(a_n) \to f(a).$
    \item $f = \elim f_n$.
     \item For every $x\in X$, $\displaystyle\lim_{n\to\infty}f_n(x)=f(x)$.
\end{enumerate}
  \end{corollary}

\begin{proof}
    Observe that since $(f_n)_{n\in\N}$ is locally equibounded, condition \eqref{equicoercividad} is satisfied (see Remark \ref{Remark_equicoerc}). By \cite[Proposition~5.12]{dalmaso1993introduction},  $f = \elim f_n$ if and only if $(f_n)_{n\in\N}$ converges pointwise to $f$, so it is enough to apply Theorem \ref{Theo:equival}.
\end{proof}

In the setting of a reflexive Banach space $X$, the roles of $X$ and $X^\ast$ can be interchanged in the above result.
Hence, the equicoercivity assumption \eqref{equicoercividad} can be formulated in the primal space.
Moreover, since in this case the $w^\ast$-topology on $X^\ast$ coincides with the weak topology, we obtain the following corollaries as a direct consequence of Theorem~\ref{Theo:equival} and Corollary \ref{Cor:pontwiseconvergence}, respectively.

\begin{corollary}\label{dualidad_reflex}
	  Let $X$ be a reflexive Banach space, and let $f, f_n\colon X \to \mathbb{R} \cup \{\infty\}$, $n \in \mathbb{N}$, be proper lower semicontinuous convex functions. 
	  Consider the following statements:
	  \begin{enumerate}[label=(\alph*)]
	  	\item $f = \wwelim f_n $.
	  	\item $ \partial f_n\xrightarrow{G(w\times \| \cdot\|_{\ast})} \partial f $ and there exist  $(a, a^*) \in \partial f$  and a sequence  $(a_n, a_n^*) \in \partial f_n$ such that 
	   $a_n \rightharpoonup a$, $ \dual{a}_n \to \dual{a} $  and  $f_n(a_n) \to f(a)$.
       \item $	\dual{f}= \elim \dual{f}_n.$
	  \end{enumerate}

      Then the implication $(a)\Longrightarrow(b)\Longrightarrow(c)$ holds. Moreover, if we assume that the sequence $(f_n)_{n\in\N}$ satisfies 
	  \begin{align}\label{equicoercividad_primal}
	    (x_n)_{n\in\N} \subseteq X, \quad \sup_{n \in \mathbb{N}} f_n(x_n) < \infty \implies \sup_{n \in \mathbb{N}} \|x_n\| < \infty.
	  \end{align}
      Then the equivalence $(a)\Longleftrightarrow(b)\Longleftrightarrow(c)$ holds.
\end{corollary}

\begin{corollary}\label{Cor:pointwiseconvergence_reflex}
     Let $X$ be a reflexive space, and let $f, f_n\colon X\colon \to \mathbb{R} \cup \{\infty\}$, $n \in \mathbb{N}$, be proper lower semicontinuous convex functions. In addition, suppose that $(\dual{f}_n)_{n\in\N}$ is locally equibounded, i.e, for every $\dual{x}\in \dual{X}$, there exists $r >0$ and $M \in\R$ such that 
      \begin{equation*}
          \forall n\in\N, \forall \dual{y}\in \dual{\mathbb{B}}_{r}(\dual{x}),\; |\dual{f}_n(\dual{y})|\leq M,
      \end{equation*}
where $\dual{\mathbb{B}}_{r}(\dual{x})$ denotes the open ball of center $\dual{x}$ and radius $r$ with the norm $\norm{\cdot}_{\ast}$. Then, the following statements are equivalent:
\begin{enumerate}[label=(\alph*)]
    \item $f = \wwelim f_n$.
    \item$ \partial f_n\xrightarrow{G(w\times \| \cdot\|_{\ast})} \partial f $ and there exist  $(a, a^*) \in \partial f$  and a sequence  $(a_n, a_n^*) \in \partial f_n$ such that 
	   $a_n \rightharpoonup a$, $ \dual{a}_n \to \dual{a} $  and  $f_n(a_n) \to f(a)$.
    \item $\dual{f} = \elim \dual{f}_n$.
     \item For every $\dual{x}\in \dual{X}$, $\displaystyle\lim_{n\to\infty}\dual{f}_n(\dual{x})=\dual{f}(\dual{x})$.
\end{enumerate}
    
\end{corollary}

{ \color{red}
We also get the following corollary, which relates the $G(\norm{\cdot}\times\dual{w})$, the $ G(w\times\norm{\cdot}_\ast)$ and the $G(\norm{\cdot}\times\norm{\cdot}_\ast)$ convergence of subdifferentials in the reflexive case.
\begin{corollary}\label{corolario}
	Let $X$ be a reflexive Banach space, and let $f, f_n\colon X \to \mathbb{R} \cup \{\infty\}$, $n \in \mathbb{N}$, be proper lower semicontinuous convex functions. Consider the following statements:
    \begin{enumerate}[label=(\alph*)]
        \item $\partial f_n \xrightarrow{G(\|\cdot\| \times \dual{w})} \partial f$ and there exist $(a,a^*) \in \partial f$ and $(a_n,a_n^*) \in \partial f_n$ such that $a_n \to a$, 
  $a_n^* \overset{\ast}{\rightharpoonup} a^*$ and  
    $f_n(a_n) \to f(a)$.
    \item $ \partial f_n\xrightarrow{G(w\times \| \cdot\|_{\ast})} \partial f $ and there exist  $(a, a^*) \in \partial f$  and a sequence  $(a_n, a_n^*) \in \partial f_n$ such that 
	   $a_n \rightharpoonup a$, $ \dual{a}_n \to \dual{a} $  and  $f_n(a_n) \to f(a)$.
       \item $\partial f_n \xrightarrow{ G(\norm{\cdot} \times\norm{\cdot}_\ast)} \partial f$ and   there exist  $(a, a^*) \in \partial f$  and a sequence  $(a_n, a_n^*) \in \partial f_n$ such that 
	   $a_n \to a$, $ \dual{a}_n \to \dual{a} $  and  $f_n(a_n) \to f(a)$.
       \item $	\dual{f}= \elim \dual{f}_n$ and $f = \elim f_n$.
        
    \end{enumerate}
If the sequence $(f_n)_{n\in\N}$ satisfies \eqref{equicoercividad_primal} or if the sequence $(\dual{f}_n)_{n\in\N}$ satisfies \eqref{equicoercividad}, then $$((a) \wedge(b))\Longleftrightarrow (c)\Longleftrightarrow (d).$$  
\end{corollary}
\begin{proof}
    That $(c)$ implies $(a)$ and $(b)$ follows directly from \cite[Proposition 3.59]{Libro_Attouch}. That $(c)$ implies $(d)$ follows from the Attouch Theorem and the fact that Mosco convergence makes the Fenchel conjugate bicontinuous in reflexive spaces. That $(a)$ and $(b)$ implies $(d)$ follows from Theorem \ref{Theo:equival} and Corollary \ref{dualidad_reflex}.
	Now, assume that the sequence $(f_n)_{n\in\N}$ satisfies \eqref{equicoercividad_primal}. Suppose that $(a)$ and $(b)$ hold. Using Theorem \ref{Theo:equival} and Corollary \ref{dualidad_reflex}, we deduce that
	$f_n$ Mosco-converges to $f$. Hence, using the Attouch theorem, we conclude $(c)$. To see that $(d)$ implies $(c)$ (asumming that $(f_n)_{n\in\N}$ satisfies \eqref{equicoercividad_primal}), we use Corollary \ref{dualidad_reflex} to deduce that $f= \wwelim f_n$, and then $f_n$ Mosco converges to $f$, so we conclude $(c)$ by the Attouch theorem.

    In order to see that $(d)$ implies $(c)$ when $(\dual{f}_n)_{n\in\N}$ satisfies \ref{equicoercividad}, we use Theorem \ref{Theo:equival} to deduce that $\dual{f} = \welim \dual{f}_n$, and therefore, $\dual{f}_n$ Mosco converges to $\dual{f}$, hence, using that the Fenchel conjugate is bicontinuous with respect to Mosco Convergence, we obtain that $f_n$ Mosco converges to $f$. Thence, we conclude $(c)$ using the Attouch theorem.
    
    %The other implication follows directly from \cite[Proposition 3.59]{Libro_Attouch}.  
\end{proof}
}

\section{Application to convex integral functionals}\label{application}

In this section we present an application of Corollary \ref{dualidad} to convex
integral functionals. Functionals of this type arise naturally in the calculus
of variations, optimal control, and stochastic optimization, and were
systematically studied by Rockafellar in \cite{Rockafellar1968IntegralsConvexFunctionals}. Let $(T, \mathcal{F},\mu)$ be a $\sigma$-finite measure space. Let $1\leq p <\infty$ and consider $L$ to be the space $L_{n}^{p}(T)$ consisting of all $\R^n$-valued measurable functions $u$ on $T$ such that

\begin{equation*}
    \int_{T}|u(t)|^p d\mu(t).
\end{equation*}

With $|\cdot|$ denoting the euclidean norm in $\R^n$. We will consider functionals of the form
\begin{equation*}
    I_f(u):= \int_{T}f(t,u(t))d\mu(t) \text{ for }u\in L,
\end{equation*}

where $f$ is a function from $T\times\R^n$ to $\R\cup\{\infty\}$ such that $f$ is convex integrand, i.e, $f(t,x)$ is a convex function of $x\in\R^n$ for each $t\in T$, hence $I_f$ is convex. Moreover, we will consider $f$ to be a \emph{normal integrand}, i.e, $f(t,\cdot)$ is proper and lower semicontinuous for every $t\in T$ and if further there exists a countable collection \(U\) of measurable
functions \(u\) from \(T\) to \(\mathbb{R}^n\) having the following properties:

\begin{enumerate}
\item for each \(u \in U\), \(f(t,u(t))\) is measurable in \(t\);

\item for each \(t\), \(U_t \cap \operatorname{dom} f_t\) is dense in \(\operatorname{dom} f_t\), where
\[
U_t = \{u(t): u \in U\}.
\]
\end{enumerate}
Although these last assumptions might seem complicated, there are some remarkable examples of normal integrands mentioned in \cite{Rockafellar1968IntegralsConvexFunctionals} and in \cite[Chapter~14]{Rockafellar_Wets}. In fact, the natural character of this class was shown in \cite{Rockafellar1969MeasurableDependence}.

\begin{example}
     Suppose \(f(t,x)=F(x)\) for all \(t\in T\), where \(F\) is a lower
semicontinuous proper convex function on \(\mathbb{R}^n\).
Then \(f\) is a normal convex integrand.
\end{example}

\begin{example}
    Suppose \(f\) is a convex integrand such that \(f(t,x)\)
is measurable in \(t\) for each fixed \(x\in\R^n\), and such that, for every \(t\in T\),
\(f(t,x)\) is lower semi-continuous in \(x\) and has interior points in its
effective domain \(\{x : f(t,x) < +\infty\}\).
Then \(f\) is a normal convex integrand. In particular, if $f$ is a convex integrand having only finite values such that $f(\cdot,x)$ is measurable for each $x$, the $f$ is a normal convex integrand.
\end{example}

\begin{example}
    Falling in the category of normal integrands are all \emph{Carathéo-\\dory integrands},
i.e., the functions \( f : T \times \mathbb{R}^n \to \mathbb{R} \) (finite-valued)
such that \( f(t,x) \) is measurable in \( t \) for each \( x \) and continuous in
\( x \) for each \( t \). Indeed, a function \( f : T \times \mathbb{R} \to \overline{\mathbb{R}} \)
is a Carathéodory integrand if and only if both \( f \) and \( -f \) are proper
normal integrands.
\end{example}

\begin{example}\label{ex:joint}
    If \( f : T \times \mathbb{R}^n \to \overline{\mathbb{R}} \) is a normal integrand,
then \( f(t,x) \) is \( \mathcal{A} \otimes \mathcal{B}(\mathbb{R}^n) \)-measurable
as a function of \( (t,x) \), in addition to being lsc as a function of \( x \).
Conversely, these properties ensure that \( f \) is a normal integrand when
\( (T,\mathcal{F},\mu) \) is complete. In
particular, when \( T = \mathbb{R}^d \) and \( \mathcal{F} = \mathcal{L}(\mathbb{R}^d) \) with $\mu$ the Lebesgue measure.
\end{example}

We recall from \cite[Lemma~5]{Rockafellar1968IntegralsConvexFunctionals} that if $f$ is a normal convex integrand then $\dual{f}$ is also a normal convex integrand and $f(t,u(t))$ is a measurable function of $t$ for every measurable function $u$ from $T$ to $\R^n$.

Lastly, we recall the following theorem and corollary from \cite[Theorem~2]{Rockafellar1968IntegralsConvexFunctionals}.

\begin{theorem}\label{Theo:dualidad_integrales}
    Let \(f\) be a normal convex integrand such that \(f(t,u(t))\) is
$\mu$-integrable for at least one \(u \in L\), and
\(f^\ast(t,u^\ast(t))\) is $\mu$-integrable for at least one
\(u^\ast \in L^\ast\).
Then \(I_f\) on \(L\) and \(I_{f^\ast}\) on \(L^\ast\) are proper convex
functions conjugate to each other.
\end{theorem}

\begin{corollary}
Suppose that \(T\) is of finite measure. Let \(f\) be of the form \(f(t,x)=F(x)\), where
\(F\) is a lower semicontinuous proper convex function on \(\mathbb{R}^n\).
Then \(I_f\) on \(L\) and \(I_{f^\ast}\) on \(L^\ast\) are conjugate to each other.
\end{corollary}

Since $\mu$ is $\sigma$-finite, the space $L^p_n(T)$ is WCG for every
$p\in[1,\infty)$, regardless of separability (which can not be ensured), and this includes the
non-reflexive case $p=1$. Consequently, Corollary \ref{dualidad}
applies to the convex integral functionals considered in this section.

\begin{corollary}
    Let $f,f_n\colon T\times\R^n \to \R\cup\{\infty\}$, $n\in\N$, be normal convex integrands. Assume that there are $u, u_n \in L$, $n\in\N$, and $\dual{u},\dual{u}_n\in\dual{L}, n\in\N$, such that \(f(t,u(t)), f_n(t,u_n(t)), \dual{f}(t,\dual{u}(t)), \dual{f}_n(t,\dual{u}_n(t))\) are
$\mu$-integrable. Suppose that $$\forall(\dual{v}_n)_{n\in\N}\subseteq \dual{L}, \sup_{n\in\N}I_{\dual{f}_n}(\dual{v}_n)<\infty \implies \sup_{n\in\N}\norm{\dual{v}_n}_{\ast}<\infty.$$

Then, $I_f = \elim I_{f_n}$ if and only if 
		$I_{\dual{f}} = \welim I_{\dual{f}_n}.$

\end{corollary}

\begin{corollary}
    Let $F,F_n\colon \R^n \to \R\cup\{\infty\}$, $n\in\N$, be proper lower semicontinuous convex functions. Assume that $T$ is of finite measure. In addition, suppose that $$\forall(\dual{v}_n)_{n\in\N}\subseteq \dual{L}, \sup_{n\in\N}I_{\dual{F}_n}(\dual{v}_n)<\infty \implies \sup_{n\in\N}\norm{\dual{v}_n}_{\ast}<\infty.$$

Then, $I_F = \elim I_{F_n}$ if and only if 
		$I_{\dual{F}} = \welim I_{\dual{F}_n}.$
\end{corollary}

When the sequence of functionals is locally equibounded, we may apply Corollary \ref{Cor:pontwiseconvergence}.
\begin{corollary}
    Let $f,f_n\colon T\times\R^n \to \R\cup\{\infty\}$, $n\in\N$, be normal convex integrands. Assume that there are $u, u_n \in L$, $n\in\N$, and $\dual{u},\dual{u}_n\in\dual{L}, n\in\N$, such that \(f(t,u(t)), f_n(t,u_n(t)), \dual{f}(t,\dual{u}(t)), \dual{f}_n(t,\dual{u}_n(t))\) are
$\mu$-integrable. Suppose that the sequence of functionals $(I_{f_n})_{n\in\N}$ is locally equibounded, i.e, for every $u\in L$, there exists $r >0$ and $M \in\R$ such that 
      \begin{equation*}
          \forall n\in\N, \forall v\in \mathbb{B}_{r}(u),\; |I_{f_n}(v)|\leq M.
      \end{equation*}

Then, the following are equivalent
\begin{enumerate}[label=(\alph*)]
\item$I_f = \elim I_{f_n}$.
\item $I_{\dual{f}} = \welim I_{\dual{f}_n}$. 
\item $I_f$ is the pointwise limit of $I_{f_n}$.
\end{enumerate}
\end{corollary}

\begin{corollary}
    Let $F,F_n\colon \R^n \to \R\cup\{\infty\}$, $n\in\N$, be proper lower semicontinuous convex functions. Assume that $T$ is of finite measure,  In addition, suppose that the sequence of functionals $(I_{f_n})_{n\in\N}$ is locally equibounded, i.e, for every $u\in L$, there exists $r >0$ and $M \in\R$ such that 
      \begin{equation*}
          \forall n\in\N, \forall v\in \mathbb{B}_{r}(u),\; |I_{f_n}(v)|\leq M.
      \end{equation*}

Then, the following are equivalent
\begin{enumerate}[label=(\alph*)]
\item$I_f = \elim I_{f_n}$.
\item $I_{\dual{f}} = \welim I_{\dual{f}_n}$. 
\item $I_f$ is the pointwise limit of $I_{f_n}$.
\end{enumerate}
\end{corollary}

In what remains of this section, we will assume that $p \in (1,\infty)$, hence, L is reflexive, therefore, besides the two previous corollaries, we can apply Corollary \ref{dualidad_reflex} to deduce the following corollaries.

\begin{corollary}
    Let $f,f_n\colon T\times\R^n \to \R\cup\{\infty\}$, $n\in\N$, be normal convex integrands. Assume that there are $u, u_n \in L$, $n\in\N$, and $\dual{u},\dual{u}_n\in\dual{L}, n\in\N$, such that \(f(t,u(t)), f_n(t,u_n(t)), \dual{f}(t,\dual{u}(t)), \dual{f}_n(t,\dual{u}_n(t))\) are
$\mu$-integrable. Suppose that $$\forall(v_n)_{n\in\N}\subseteq L, \sup_{n\in\N}I_{f_n}(v_n)<\infty \implies \sup_{n\in\N}\norm{v_n}<\infty.$$

Then, $I_f = \wwelim I_{f_n}$ if and only if 
		$I_{\dual{f}} = \elim I_{\dual{f}_n}.$

\end{corollary}

\begin{corollary}
    Let $F,F_n\colon \R^n \to \R\cup\{\infty\}$, $n\in\N$, be proper lower semicontinuous convex functions. Assume that $T$ is of finite measure. In addition, suppose that $$\forall(v_n)_{n\in\N}\subseteq L, \sup_{n\in\N}I_{F_n}(v_n)<\infty \implies \sup_{n\in\N}\norm{v_n}<\infty.$$

Then, $I_F = \wwelim I_{F_n}$ if and only if 
		$I_{\dual{F}} = \elim I_{\dual{F}_n}.$
\end{corollary}

Lastly, in the case where the conjugates are locally equibounded, we can apply Corollary \ref{Cor:pointwiseconvergence_reflex} to obtain the following corollaries.

\begin{corollary}
    Let $f,f_n\colon T\times\R^n \to \R\cup\{\infty\}$, $n\in\N$, be normal convex integrands. Assume that there are $u, u_n \in L$, $n\in\N$, and $\dual{u},\dual{u}_n\in\dual{L}, n\in\N$, such that \(f(t,u(t)), f_n(t,u_n(t)), \dual{f}(t,\dual{u}(t)), \dual{f}_n(t,\dual{u}_n(t))\) are
$\mu$-integrable. Suppose that the sequence of functionals $(I_{\dual{f}_n})_{n\in\N}$ is locally equibounded, i.e, for every $\dual{u}\in L^\ast$, there exists $r >0$ and $M \in\R$ such that 
      \begin{equation*}
          \forall n\in\N, \forall \dual{v}\in \dual{\mathbb{B}}_{r}(\dual{u}),\; |I_{\dual{f}_n}(\dual{v})|\leq M,
      \end{equation*}
where $\dual{\mathbb{B}}_{r}(\dual{u})$ is the open ball of center $\dual{u}$ and radius $r$ with the dual norm in $\dual{L}$.
Then, the following are equivalent
\begin{enumerate}[label=(\alph*)]
\item$I_{\dual{f}} = \elim I_{\dual{f}_n}$.
\item $I_{f} = \wwelim I_{f_n}$. 
\item $I_{\dual{f}}$ is the pointwise limit of $I_{\dual{f}_n}$.
\end{enumerate}
\end{corollary}

\begin{corollary}
    Let $F,F_n\colon \R^n \to \R\cup\{\infty\}$, $n\in\N$, be proper lower semicontinuous convex functions. Assume that $T$ is of finite measure. In addition,  Suppose that the sequence of functionals $(I_{\dual{F}_n})_{n\in\N}$ is locally equibounded, i.e, for every $\dual{u}\in L^\ast$, there exists $r >0$ and $M \in\R$ such that 
      \begin{equation*}
          \forall n\in\N, \forall \dual{v}\in \dual{\mathbb{B}}_{r}(\dual{u}),\; |I_{\dual{F}_n}(\dual{v})|\leq M,
      \end{equation*}
where $\dual{\mathbb{B}}_{r}(\dual{u})$ is the open ball of center $\dual{u}$ and radius $r$ with the dual norm in $\dual{L}$.
Then, the following are equivalent
\begin{enumerate}[label=(\alph*)]
\item$I_{\dual{F}} = \elim I_{\dual{F}_n}$.
\item $I_{F} = \wwelim I_{F_n}$. 
\item $I_{\dual{F}}$ is the pointwise limit of $I_{\dual{F}_n}$.
\end{enumerate}
\end{corollary}

These results provide a dual characterization of the $\Gamma$-convergence
of convex integral functionals on $L^p$ spaces in terms of the $\Gamma$-convergence of their conjugates with the $\dual{w}$-topology.

\begin{remark}
It is known from \cite{Diestel1975L1WCG} that the space of Bochner-integrable functions defined in a finite measure space with values in a Banach space Y is WCG whenever Y is WCG (one may extend this to a $\sigma$-finite measure). We also recall that the Bochner space $L^p(\mu,Y)$ for $p\in (1,\infty)$ is reflexive whenever $Y$ is reflexive (thus WCG). Therefore, one may try to apply the duality of $\Gamma$-convergence studied in Section \ref{MainResults} to the more general setting of Bochner spaces. The conjugate of integral functionals given by measurable integrands in the product $T\times Y$ has been studied by Rockafellar in the case where $Y$ is separable (hence WCG) in \cite[Theorem~21]{Rockafellar1974ConjugateDuality}, although unlike Theorem \ref{Theo:dualidad_integrales}, lower semicontinuity is not warranted by duality, so that it might have to be checked separately. Nevertheless, in the case where $Y$ is reflexive and separable, a result analogous to Theorem \ref{Theo:dualidad_integrales} is presented for convex normal integrands (where in this case, normal integrands are defined as in Example \ref{ex:joint}) in \cite[Theorem~VII-7]{CastaingValadier1977}. 
\end{remark}

\section{Some Preliminary  Results}\label{Some_Pre_results}

In this section, we collect several results that will be used in the sequel.

The first lemma states a rule for interchanging upper limits. A proof can be found in \cite[Corollary 1.16]{Libro_Attouch}.

\begin{lemma}\label{Lemma:IntercambioLimsup}
	Let $\{\alpha_{k,n}: k\geq 1,\; n\geq 1\}\subseteq \R \cup\{-\infty,\infty\}$ be a double indexed sequence of extended real numbers. Then, there is $(n_k)_{k\in\N}$ with $n_k\to \infty$ as $k\to\infty$ such that
	\begin{equation*}%\label{apendice}
		\limsup_{k\to\infty}\alpha_{k,n_k}\leq \limsup_{n\to\infty}\limsup_{k\to\infty}\alpha_{k,n}.    
	\end{equation*}
\end{lemma}

The second lemma provides sufficient conditions for the continuity of the Fenchel conjugate with respect to the norm topology in the dual space. This result follows directly from \cite[Corollary~4.4.11]{Libro_borwein}.

\begin{lemma}\label{cont_conj}
	Let $f\colon  X \to \R \cup \{\infty\}$ be a convex function. 
	If there exist $\alpha > 0$, $\beta \in \R$, and $\bar{x}^{\ast} \in X^{\ast}$ such that for all $x \in X$
	\begin{equation*}
		f(x) \geq \alpha \|x\|^2 + \langle \bar{x}^{\ast}, x \rangle + \beta,
	\end{equation*}
	then $\dual{f}$ is continuous with respect to the norm topology.
\end{lemma}

The next lemma concerns the approximation of a convex proper lower semicontinuous function by its Moreau envelope, together with its continuity properties. We refer to \cite[Theorem 3.265]{MR4659162} and \cite[Theorem 4.4.17]{Libro_borwein} for more details.

\begin{lemma}\label{ConvMoreau}
	Let $f\colon X \to \R \cup \{\infty\}$ be a proper lower semicontinuous convex function. Then, for every $x \in X$, one has 
	\[
	f_{\lambda}(x) \nearrow f(x) \quad \text{as } \lambda \to 0,
	\]
	and $f_\lambda$ is continuous for every $\lambda > 0$.
\end{lemma}

The following lemma establishes an approximation principle for conjugate functions via infimal convolutions. Its proof relies on the same ideas as the classical approximation of functions by infimal convolution, and is therefore omitted.

\begin{lemma}\label{aproximacion}
	Let $f\colon X \to \R \cup \{\infty\}$ be a proper lower semicontinuous convex function, and let $A \subseteq X$ be a weakly compact, convex, and symmetric set such that the space generated by $A$ is dense in $X$ with the norm topology. In addition, suppose that $f$ is continuous at $0$. Then, for every $\dual{x} \in X^{\ast}$,
	\[
	g_n(\dual{x}) := \dual{f} \,\square\, \sigma_{[0,n]A}(\dual{x}) \nearrow \dual{f}(\dual{x}) \quad \text{as } n \to \infty,
	\]
	and $g_n$ is continuous with respect to the Mackey topology for every $n\in\N$. 
\end{lemma}

\section{Proof of Theorem \ref{previo_a_attouch}}\label{Section:Proof_previo_a_attouch}

In this section, we provide a proof of Theorem~\ref{previo_a_attouch}.  
The result is first established for separable Banach spaces and then extended to general WCG Banach spaces by means of separable reduction techniques, supported by the use of projectional skeletons.

We begin by formally establishing a relation that always holds. This is stated in the following lemma.

\begin{lemma}\label{Lemma:inq_01conjugate}
	Let $X$ be a Banach space, and let $f_n\colon X\to\R\cup\{\infty\}$, $n\in\N$, be a sequence of lower semicontinuous proper convex functions.  
	Define $f := \eli f_n$. Then, one always has
	\begin{equation*}%\label{inq_01conjugate}
		f^\ast \leq \wels f_n^\ast.
	\end{equation*}
\end{lemma}
\begin{proof}
	Let $\dual{x} \in \dual{X}$ and $(\dual{x}_n)_{n \in \N} \subseteq \dual{X}$ be such that $\dual{x}_n \overset{\ast}{\rightharpoonup} \dual{x}$.  
Fix $x \in X$ and a sequence $(x_n)_{n \in \N} \subseteq X$ with $x_n \to x$.  
By the definition of the conjugate, for every $n \in \N$ we have
\begin{equation*}
    \langle \dual{x}_n, x_n \rangle - f_n(x_n) \leq f_n^\ast(\dual{x}_n).
\end{equation*}

Taking the $\limsup$, we obtain
\begin{align*}
    \limsup_{n \to \infty} \big( \langle \dual{x}_n, x_n \rangle - f_n(x_n) \big) 
    &= \langle \dual{x}, x \rangle + \limsup_{n \to \infty} \big( - f_n(x_n) \big) \\
    &= \langle \dual{x}, x \rangle - \liminf_{n \to \infty} f_n(x_n) \\
    &\leq \limsup_{n \to \infty} f_n^\ast(\dual{x}_n).
\end{align*}

Since the choice of $(\dual{x}_n)_{n\in\N}$ with $\dual{x}_n \overset{\ast}{\rightharpoonup} \dual{x}$ was arbitrary, it follows that
\begin{equation*}
    \langle \dual{x}, x \rangle - \liminf_{n \to \infty} f_n(x_n) \;\leq\; \wels f_n^\ast(\dual{x}).
\end{equation*}

Moreover, since $(x_n)_{n \in \N} $ was also an arbitrary sequence converging to $x$, we deduce that
\begin{equation*}
    \langle \dual{x}, x \rangle - \eli f_n(x)\;\leq\;  \wels f_n^\ast(\dual{x}).
\end{equation*}

Equivalently,
\begin{equation*}
    \langle \dual{x}, x \rangle - f(x) \;\leq\; \wels f_n^\ast(\dual{x}).
\end{equation*}

Because this inequality holds for every $x \in X$, we conclude that
\[
    f^\ast(\dual{x}) \leq \wels f_n^\ast(\dual{x}),  \text{ for all } \dual{x} \in \dual{X},
\]
which completes the proof of the lemma.

\end{proof}

\subsection{Proof of Theorem \ref{previo_a_attouch} under separability of $X$}\label{Proof_previo_a_attouch_Separable}
In this subsection, we establish Theorem~\ref{previo_a_attouch} in the case where $X$ is separable.  

\begin{proof}
Suppose that $X$ is a separable Banach space and denote $f := \eli f_n$.  
From Lemma~\ref{Lemma:inq_01conjugate}, we already know that
\[
f^\ast \;\leq\; \wels f_n^\ast.
\]
It remains to prove the reverse inequality in the separable case, namely,
\[
f^\ast \;\geq\; \wels f_n^\ast.
\]

For that purpose, it is sufficient to show that  that for every $ x^* \in X^*$  there exists a sequence $(z_n^*)_{n\in\N}$ such that $z_n^* \overset{\ast}{\rightharpoonup} x^*$ and
\begin{equation}\label{final_inqu}
f^*(x^*) \geq \limsup_n f_n^*(z_n^*).
\end{equation}
 
Let $\{x_1, x_2, \dots\}$ be a countable dense subset of $X$.
Take $a \in \dom(\els f_n)$, and let $(a_n)_{n\in\N}$ be a sequence such that
\begin{align}\label{uniformemente_propias}
a_n \to a
\quad \text{and} \quad
\limsup_{n \to \infty} f_n(a_n) = \els f_n(a) < \infty.
\end{align}
In particular, the sequence $(f_n(a_n))_{n\in\N}$ may be assumed to be bounded from above.

Let $E_k$ be the vector space generated by $\{ x_1, \dots, x_k \}$. Define
\[
A_k = \{x \in E_k : \|x\| \leq k\}, \quad K := \{a_n\}_{n \in \mathbb{N}} \cup \{a\},
\]
so that $ A_k \cup K$ is compact with respect to the norm topology. Then set
\[
B_k := \overline{\operatorname{co}}(A_k \cup K).
\]
Observe that for each $k \in \mathbb{N}$, the set $B_k$ is convex  and compact (see, e.g., \cite[Theorem 5.35]{aliprantis}).

We define the following family of functions:
\begin{align}\label{deffnklambda}
    f_{n,k,\lambda} &= f_n + \delta_{B_k} + \frac{\lambda}{2} \|\cdot\|^2, \quad 
    f_{k,\lambda} = f + \delta_{B_k} + \frac{\lambda}{2} \|\cdot\|^2.
\end{align}
Since $\dom f_n \cap B_k \neq \emptyset$ and $\dom f \cap B_k \neq \emptyset$, all functions defined in \eqref{deffnklambda} are proper.

Now, we consider $x^\ast\in X^\ast$ fixed, but arbitrary. We aim to show \eqref{final_inqu} for some $\dual{z}_n \overset{\ast}{\rightharpoonup}\dual{x}$.
Observe that the case $\dual{f}(\dual{x})=\infty$ is trivial. Note also that because of \eqref{uniformemente_propias}, $f$ is not identically $\infty$, so the case $\dual{f}(\dual{x})=-\infty$  is not possible. This said, from now on, we will assume that $\dual{f}(\dual{x})\in\R$.

The remainder of the proof is divided into five claims.

\textbf{Claim 1:} \textit{For every $k \in \mathbb{N}$ and any $\lambda >0$  we have
\begin{equation} \label{eq1}
      \quad f_{k,\lambda}^*(x^*) \geq \limsup_n f_{n,k,\lambda}^*(x^*).
\end{equation}}
Indeed, we notice that \eqref{eq1}  is equivalent to 
\begin{align*}
    \liminf_n \inf_{u \in B_k} \left\{f_n(u) + \frac{\lambda}{2}\|u\|^2 - \langle x^*, u \rangle \right\} 
    \geq \inf_{u \in B_k} \left\{f(u) + \frac{\lambda}{2}\|u\|^2 - \langle x^*, u \rangle \right\}.
\end{align*}

If the left-hand side is equal to $\infty$  the inequality is trivial. Then, let us suppose that the  left-hand side is finite. Since $B_k$ is compact, we can find $u_n\in B_k$ such that 

\begin{equation*}
    \inf_{u \in B_k} \left\{f_n(u) + \frac{\lambda}{2}\|u\|^2 - \langle x^*, u \rangle \right\}  = f_n(u_n) + \frac{\lambda}{2}\|u_n\|^2 - \langle x^*, u_n \rangle 
\end{equation*}
Furthermore, by    definition of limit inferior,  we can find a subsequence $u_{n_j} \in B_{k}$ such that:
\begin{align*}
    \lim_j \left\{f_{n_j}(u_{n_j}) + \frac{\lambda}{2}\|u_{n_j}\|^2 - \langle x^*, u_{n_j} \rangle \right\} = \liminf_n \inf_{u \in B_k} \left\{f_n(u) + \frac{\lambda}{2}\|u\|^2 - \langle x^*, u \rangle \right\}.
\end{align*}
Moreover, again using the compactness of $B_k$, by passing to a subsequence, we can assume that  $u_{n_j}$   converges (strongly) in $B_k$ to some $\bar{u} \in B_k$. We define the sequence 
\begin{equation*}
    u'_k = \begin{cases}
         u_{n_j}  &\text{ if } k=n_j,\; \text{for some } j\in\N.\\
         \bar{u} &\text{ if } k\neq n_j,\;\forall j\in\N.
     \end{cases}
    \end{equation*}

Observe that $u'_k\to \bar{u}$. Now, using that $f = \eli f_n$, we obtain that
\begin{align*}
      \lim_j \left\{f_{n_j}(u_{n_j}) + \frac{\lambda}{2}\|u_{n_j}\|^2 - \langle x^*, u_{n_j} \rangle \right\} 
      &\geq \liminf_{j}f_{n_j}(u_{n_j}) +  \frac{\lambda}{2}\|\bar{u}\|^2 - \langle x^*, \bar{u} \rangle \\
      &\geq  \liminf_{j}f_{j}(u'_{j}) +  \frac{\lambda}{2}\|\bar{u}\|^2 - \langle x^*, \bar{u} \rangle\\
      &\geq 
      f(\bar{u})  + \frac{\lambda}{2}\|\bar{u}\|^2 - \langle x^*,\bar{u} \rangle\\
      &\geq  \inf_{u \in B_k} \left\{f(u) + \frac{\lambda}{2}\|u\|^2 - \langle x^*, u \rangle \right\}.
\end{align*}
 
which shows that \eqref{eq1} holds.

\textbf{Claim 2:}
\textit{We have that for all $k \in \N$, all $\lambda>0$, and any $n \in \N$:
\begin{align}\label{formula01}
f_{n,k,\lambda}^{\ast} 
   = \operatorname{cl}((f_{n}^{\ast})_{\lambda}\square \delta_{B_{k}}^{\ast}), 
\end{align}
where the closure is taken with respect to the weak$^\ast$ topology, that is,  $$\cl((f_{n}^{\ast})_{\lambda}\square \delta_{B_{k}}^{\ast})(z^{\ast}):= \inf\{\displaystyle\liminf_{d} f(z_d): D \text{ a directed set, } (z_d)_{d\in D} \text{ a net such that } z_d \overset{\ast}{\rightharpoonup}\dual{z}\},$$ for every $\dual{z} \in \dual{X}$}\\
 First, by \cite[Theorem 3]{rockafellar1966extension}, we have
\begin{align}\label{eqconjugate}
  \left(f_{n} + \frac{\lambda}{2} \norm{\cdot}^2\right)^{\ast} =   f^\ast_{n} \square\left(  \frac{\lambda}{2} \dualnorm{\cdot}^2\right)^{\ast}  =    \left(f_n^\ast\right)_\lambda.
\end{align}
where in the last equality we have used that $\left(  \frac{\lambda}{2} \norm{\cdot}^2\right)^{\ast} = \frac{1}{2\lambda } \dualnorm{\cdot }^2$. %(see, e.g., \cite{}). 
Next, due to the fact that  $\dom \left(f_n + \frac{\lambda}{2} \norm{\cdot}^2\right) \cap \dom \delta_{B_k} \neq \emptyset$ (recall that $a_n \in B_k$), we may apply \cite[Proposition 4.1.16]{Libro_Correa} to deduce
\begin{align}\label{eqconjugate2}
  f_{n,k,\lambda}^{\ast} 
  = \left(f_n + \frac{\lambda}{2} \norm{\cdot}^2 + \delta_{B_k}\right)^\ast 
  = \operatorname{cl} \left( \left(f_n + \frac{\lambda}{2} \norm{\cdot}^2\right)^{\ast} \square \delta_{B_k}^\ast \right).
\end{align}

Hence, combining \eqref{eqconjugate} and \eqref{eqconjugate2}, we conclude that \eqref{formula01} holds.

\textbf{Claim 3:}
 \textit{There is a positive sequence $(\lambda_n)_{n\in\N}$ converging to $0$, a sequence of natural numbers $(k_n)_{n\in\N}$ such that $ k_n\to \infty$ and a sequence $(\dual{u}_n)_{n\in\N}\subseteq \dual{X}$ with $$\displaystyle\max_{1\leq j\leq n} |\inner{x^{\ast} - u^{\ast}_{n}}{x_j}| \leq \frac{1}{n}$$ such that}

    \begin{align}\label{claim 2.1}
        f^{\ast}(x^{\ast}) \geq \limsup_{n}(f_{n}^{\ast})_{\lambda_n}\square \delta_{B_{k_n}}^{\ast}(u^{\ast}_{n}).
    \end{align}

Indeed,  since $f_{k,\lambda} \geq f$, it follows that $f^* \geq f_{k,\lambda}^*$. Combining this with \eqref{eq1} and using $\lambda(k) :=\frac{1}{k}$, we get
\begin{align*}
    f^*(x^*) \geq \limsup_n f_{n,k,\lambda(k)}^*(x^*).
\end{align*}
Therefore
\begin{align*}
    f^*(x^*) \geq \limsup_k \limsup_n f_{n,k,\lambda(k)}^*(x^*).
\end{align*}
Now we use the technical result provided in Lemma~\ref{Lemma:IntercambioLimsup} 
to deduce that there is a sequence of natural numbers $(k_n)_{n\in\N}$ such that $k_n \to \infty$ and

\begin{equation}\label{desigualdad claim 3.1}
    f^*(x^*) \geq \limsup_n f_{n,k_n,\lambda(k_n)}^*(x^*).
\end{equation}

Now, using  \eqref{formula01} and the fact that $f^\ast_{n,k,\lambda}$ is proper, we have that   $$f_{n,k,\lambda}^{\ast}(x^{\ast}) + \frac{1}{n} > \cl_{\dual{w}}((f_{n}^{\ast})_{\lambda}\square \delta_{B_{k}}^{\ast})(x^{\ast}).$$

It follows that there is a directed set $I_n$ and a net $(\dual{u}_{i,n,k,\lambda})_{i\in I_n}$ convergent to $\dual{x}$ such that \begin{equation*}
     f_{n,k,\lambda}^{\ast}(x^{\ast}) + \frac{1}{n} > \liminf_{i\in I _n}(f_{n}^{\ast})_{\lambda}\square \delta_{B_{k}}^{\ast}(\dual{u}_{i,n,k,\lambda})
\end{equation*}

So we can construct a sequence $(u^{\ast}_{n,k,\lambda})_{n\in\N}\subseteq X^{\ast}$ such that

\begin{align*}
    f_{n,k,\lambda}^{\ast}(x^{\ast}) + \frac{1}{n} &> (f_{n}^{\ast})_{\lambda}\square \delta_{B_{k}}^{\ast}(u^{\ast}_{n,k,\lambda}).\\
    \max_{1\leq j\leq n} |\inner{x^{\ast} - u^{\ast}_{n,k,\lambda}}{x_j}| &\leq \frac{1}{n}.
\end{align*}

We can apply this to $k=k_n,\; \lambda = \lambda_n:= \lambda(k_n)$, and denoting by $u^{\ast}_{n} := u_{n,k_n,\lambda_n}^{\ast}$, we combine the latter with \eqref{desigualdad claim 3.1} to obtain that

\begin{align*}
    f^{\ast}(x^{\ast}) \geq \limsup_{n}(f_{n}^{\ast})_{\lambda_n}\square \delta_{B_{k_n}}^{\ast}(u^{\ast}_{n}),
\end{align*}

with $\displaystyle\max_{1\leq j\leq n} |\inner{x^{\ast} - u^{\ast}_{n}}{x_j}| \leq \frac{1}{n}$.

\textbf{Claim 4:}
    \textit{There are bounded sequences $(x_n^\ast)_{ n \in \N}$ and $(\dual{y}_n)_{n \in \N}$ such that $(\dual{f}_n(\dual{y}_n))_{n\in\N}$ is bounded below, $\norm{\dual{y}_n-\dual{x}_n}_\ast \to 0$ and}
    \begin{equation}\label{exact_inf-convolution}
      \dual{f}(\dual{x})  \geq \limsup_{n} f_{n}^{\ast}(y^{\ast}_{n})+ \frac{1}{2\lambda_n}\dualnorm{x_{n}^{\ast}-y^{\ast}_{n}}^2
    \geq \limsup_{n} f_{n}^{\ast}(y^{\ast}_{n}).
    \end{equation}

Indeed, using the inf-convolution of \eqref{claim 2.1}, we obtain that  for each $n\in\N$  there exists some $\dual{x}_n \in\dual{X}$ such  that $$(f_{n}^{\ast})_{\lambda_n}\square \delta_{B_{k_n}}^{\ast}(u^{\ast}_{n})+\frac{1}{n}\geq (f_{n}^{\ast})_{\lambda_n}(x_{n}^{\ast})+ \delta_{B_{k_n}}^{\ast}(u_{n}^{\ast} - x_{n}^{\ast}).$$

Then, using \eqref{claim 2.1} and the fact that for every $n\in\N$ there is $\dual{y_n}\in\dual{X}$ such that $$\frac{1}{n}+\displaystyle\inf_{u^{\ast}\in X^{\ast}}\{f_{n}^{\ast}(u^{\ast})+ \frac{1}{2\lambda_n}\dualnorm{x_{n}^{\ast}-u^{\ast}}^2\}\geq f_{n}^{\ast}(y^{\ast}_{n})+ \frac{1}{2\lambda_n}\dualnorm{x_{n}^{\ast}-y^{\ast}_{n}}^2,$$ we deduce that

\begin{equation}\label{2}
\begin{aligned}
    f^{\ast}(x^{\ast}) &\geq \limsup_{n} \left(  (f_{n}^{\ast})_{\lambda_n}(x_{n}^{\ast})+ \delta_{B_{k_n}}^{\ast}(u_{n}^{\ast} - x_{n}^{\ast}) \right)\\
    &= \limsup_{n}\left( \delta_{B_{k_n}}^{\ast}(u_{n}^{\ast} - x_{n}^{\ast})+\inf_{u^{\ast}\in X^{\ast}}\{f_{n}^{\ast}(u^{\ast})+ \frac{1}{2\lambda_n}\dualnorm{x_{n}^{\ast}-u^{\ast}}^2\}\right)\\
    &\geq \limsup_{n} \left(\delta_{B_{k_n}}^{\ast}(u_{n}^{\ast} - x_{n}^{\ast})+  f_{n}^{\ast}(y^{\ast}_{n})+ \frac{1}{2\lambda_n}\dualnorm{x_{n}^{\ast}-y^{\ast}_{n}}^2\right)\\
    &\geq \limsup_{n}\left(\delta_{B_{k_n}}^{\ast}(u_{n}^{\ast} - x_{n}^{\ast})+ f_{n}^{\ast}(y^{\ast}_{n})\right)\\
    &\geq \limsup_{n} f_{n}^{\ast}(y^{\ast}_{n}).
\end{aligned}
\end{equation}

Since $\dual{f}(x^{\ast})$ is finite, $(f_{n}^{\ast}(y^{\ast}_{n}))_{n \in \N}$ is bounded above, and by the property of equicoercivity \eqref{equicoercividad}, we conclude that $(y_{n}^{\ast})_{n \in \N}$ is bounded.\\

Furthermore, let us note that $(f_{n}^{\ast}(y^{\ast}_{n}))_{ n \in  \N}$ is bounded below. In fact, using the sequence $(a_n)_{n \in \N}$ from (\ref{uniformemente_propias}):

\begin{equation*}
    \begin{split}
        \liminf_{n} f_{n}^{\ast}(y^{\ast}_{n}) & \geq \liminf_{n}{\inner{y_{n}^{\ast}}{a_n} - f_{n}(a_n)}\\
        &\geq \liminf_{n}{(-\dualnorm{y_{n}^{\ast}}\norm{a_n} - f_{n}(a_n))}\\
        &= -\limsup_{n}({\dualnorm{y_{n}^{\ast}}\norm{a_n} +f_{n}(a_n)}).
    \end{split}
\end{equation*}

Since the sequences $(f_{n}(a_n))_{n\in\N}$ and $(\norm{y_{n}^{\ast}}\norm{a_n})_{n\in\N}$ are bounded above, we then get \begin{equation*}
    \displaystyle\liminf_{n} f_{n}^{\ast}(y^{\ast}_{n})> -\infty  
\end{equation*}
and the sequence $(f_{n}^{\ast}(y^{\ast}_{n}))_{n\in\N}$ is bounded below.

Then, returning to \eqref{2}, since $\lambda_n\to 0$ and $(f_{n}^{\ast}(y^{\ast}_{n}))_{n\in\N}$ is bounded below, it follows that $\dualnorm{x_{n}^{\ast}-y^{\ast}_{n}}\to 0$. Thus, since the sequence $(y_{n}^{\ast})_{n\in\N}$ is bounded, it follows that $(x_{n}^{\ast})_{n\in\N}$ is also bounded.\\

\textbf{Claim 5:}\label{claim 5.1}
    \textit{The sequence $(\dual{y}_n)_{n\in\N}$ from \eqref{exact_inf-convolution} satisfies $\dual{y}_n \overset{\ast}{\rightharpoonup} x^{\ast}$.}

We will show that the sequence $(\dual{x}_n)_{n\in\N}$ from \eqref{exact_inf-convolution} satisfies $x_{n}^{\ast}  \overset{\ast}{\rightharpoonup} x^{\ast}$. In fact, using \eqref{2} and that the sequence $(f_{n}^{\ast}(y^{\ast}_{n}))_{n\in\N}$ is bounded below, we obtain that

\begin{equation} \label{eqClaim5_3}
\begin{aligned}
  f^{\ast}(x^{\ast}) &\geq \limsup_{n} f_{n}^{\ast}(y^{\ast}_{n}) + \delta_{B_{k_n}}^{\ast}(u_{n}^{\ast} - x_{n}^{\ast})\\
  &\geq \inf_{n \in \N } f^\ast (y^\ast_n)  + \limsup_{n}\delta_{A_{k_n}}^{\ast}(u_{n}^{\ast} - x_{n}^{\ast}).
\end{aligned}
\end{equation}

Note that by  \cite[Proposition 4.1.20]{Libro_Correa}, we have
\begin{equation*}
\begin{split}
    \delta_{A_{k_n}}^{\ast} &= (\delta_{E_{k_n}}+ \delta_{\{x\in X:\norm{x}\leq k_n\}})^{\ast} = \delta_{E_{k_n}}^{\ast}\square \delta_{\{x\in X:\norm{x}\leq k_n\}}^{\ast}= \delta_{E_{k_n}^{\perp}} \square k_n\norm{\cdot}_{\ast} =  k_nd(x^{\ast},E_{k_n}^{\perp}),
    \end{split}
\end{equation*}
where $d(x^{\ast},E_{k_n}^{\perp})$ denotes the distance from $x^\ast$ to $E_{k_n}^{\perp}$.

Using this and \eqref{eqClaim5_3}, we obtain that 
\begin{equation*}
f^{\ast}(x^{\ast})-\inf_{n \in \N } f^\ast (y^\ast_n)  \geq \displaystyle\limsup_{n}{k_nd(u_{n}^{\ast} -x_{n}^{\ast},E_{k_n}^{\perp})}.
\end{equation*}
 Since  $k_n\to \infty$, it follows that
\begin{equation*} 
    d(u_{n}^{\ast} -x_{n}^{\ast},E_{k_n}^{\perp}) \to 0.
\end{equation*}
 
Then, given $\varepsilon>0$, there is $n_0 \in \N$ such that for all $n\geq n_0$ there is $\dual{h}_n \in E_{k_n}^{\perp}$ such that

\begin{equation} \label{4}
   \norm{\dual{u}_n -\dual{x}_n-\dual{h}_n}_{\ast} \leq \varepsilon.
\end{equation}

Let $D := \displaystyle\bigcup_{k\in\N} E_{k}$, $D$ is dense. Let $x\in D$; since $D$ is an increasing union, it follows that $x\in E_{k_n}$ for sufficiently large $n$.\\

We have that $\inner{u_{n}^{\ast} -x_{n}^{\ast}}{x}= \inner{u_{n}^{\ast} -x_{n}^{\ast}- \dual{h}_n}{x}$.\\

Consequently, $\abs{\inner{u_{n}^{\ast} -x_{n}^{\ast}}{x}}\leq \norm{\dual{u}_n -\dual{x}_n-\dual{h}_n}_{\ast}\norm{x}\leq \varepsilon\norm{x}$. Which, by \eqref{4}, implies that 

\begin{equation*}
 \inner{u_{n}^{\ast} -x_{n}^{\ast}}{x}\to 0 \text{ for all } x \in D.
\end{equation*}

Let us prove that $x_{n}^{\ast}  \overset{\ast}{\rightharpoonup} x^{\ast}$. Since $(x_{n}^{\ast})_{n\in\N}$ is norm bounded and $\{x_j:j\in\N\}$ is dense in $X$, all we need to prove is that 
\begin{equation*}
     \inner{\dual{x}_n-\dual{x}}{x_j}\to 0 \text{ for all } j \in  \N.
\end{equation*}

So let $j\in\N$ and $j\geq n$, then
\begin{equation*}
    |\inner{\dual{x}_n -\dual{x}}{x_j }| = |\inner{\dual{x}_n -\dual{u}_n}{x_j } + \inner{\dual{u}_n -\dual{x}}{x_j } |\leq  |\inner{\dual{x}_n -\dual{u}_n}{x_j }| + \frac{1}{n}.
\end{equation*}

It follows that $|\inner{\dual{x}_n -\dual{x}}{x_j }|\to 0$ and this is valid for every $j\in\N$,  hence $x_{n}^{\ast}  \overset{\ast}{\rightharpoonup} x^{\ast}$.\\

Finally, since $\dualnorm{x_{n}^{\ast}-y^{\ast}_{n}}\to 0$, it is concluded that 

\begin{equation}\label{convergencia}
    y_{n}^{\ast}  \overset{\ast}{\rightharpoonup} x^{\ast}.
\end{equation}

Since by \eqref{exact_inf-convolution} and \eqref{convergencia} the sequence $(\dual{y}_n)_{n\in\N}$ satisfies $\dual{y}_n \overset{\ast}{\rightharpoonup} \dual x$ and also $\dual{f}(\dual{x})\geq \displaystyle\limsup_{n}\dual{f}_n(\dual{y}_n)$, it is enough to choose $\dual{z}_n=\dual{y}_n$ for each $n\in\N$ in \eqref{final_inqu}. The proof in the case when X is assumed separable is concluded.
\end{proof}
 \subsection{Separable reduction and rich families}\label{Sec_Sep_red_rich}
 We say that a property \(\mathcal{P}\) concerning a Banach space \(X\), is \emph{separably reducible} (or \emph{separably determined}) if there exists a family \(\mathcal{F}\) of separable subspaces of \(X\) such that   $\cup\left\{ F  : F  \in \mathcal{F}  \right\}     = X$  
 and  the property $\mathcal{P}$ holds in $X$  if and only if for every $  F\in\mathcal{F} $ the property $\mathcal{P}$ holds at $F$.

 Separable reducibility is a common tool in functional analysis, as it often leads to substantial simplifications and, in some cases, provides the only viable approach to proving certain results. A simple example of a separably reducible property is continuity of
 a mapping. For a list of some less trivial properties that are separable determined in analysis, we refer the reader to \cite{fabian_separable_determined}.
 There are special families of separable subspaces of a given Banach space that allow the implementation of a \emph{separable reduction}, that is, the possibility to “lift” a property from separable subspaces to the entire space, since they allow one to deal with countable many different hypotheses. These special families are called rich families, and they were introduced by J.M Borwein and W.B Moors in \cite{Borwein}. Rich families were deeply discussed in connection with a number of
 properties related to differentiability and subdifferentiability in \cite{annals} and \cite{rich_families_subdifferential}.
 \begin{definition}
 	Let $Z$ be a (rather) non-separable metric (or normed) space. Denote by $\mathcal{S}(Z)$ the family of all closed separable subsets (or closed separable subspaces) of $Z$. A subfamily $R \subset \mathcal{S}(Z)$ is called \textit{rich} if:
 	
 	\begin{enumerate}
 		\item[(i)] For every $S \in \mathcal{S}(Z)$, there exists $Y \in R$ such that $S \subset Y$.
 		\item[(ii)] Whenever $Y_1 \subset Y_2 \subset \cdots \in R$, then $Y := \adh{\displaystyle\bigcup_{n\in\N} Y_n} \in R$.
 	\end{enumerate}
 \end{definition}
 The main property and the power of rich families is shown in the following fundamental fact (see \cite[Proposition 1.1]{Borwein}).
 
 \begin{proposition}\label{rich families}
 	The intersection of countably many rich families of a given partial order is (not only non-empty but also) rich.
 \end{proposition}
 Let $k \in \mathbb{N}$ be greater than $1$, and let $X_1, \dots, X_k$ be Banach spaces. 
 A \emph{block} is any product of the form $Y_1 \times \cdots \times Y_k$. 
 A \emph{block-family} in $\mathcal{S}(X_1 \times \cdots \times X_k)$ (or simply in $X_1 \times \cdots \times X_k$) 
 is any family consisting of such blocks, where $Y_i \in \mathcal{S}(X_i)$ for $i = 1, \dots, k$.
 
 In the special case $k=2$, we refer to blocks as \emph{rectangles} and to block-families as \emph{rectangle-families}. 
 We denote by $
 \mathcal{S}_{\rectangle{}}(X_1 \times X_2)$
 the maximal rectangle-family in $\mathcal{S}(X_1 \times X_2)$. 
 Clearly, $\mathcal{S}_{\rectangle{}}(X_1 \times X_2)$ is a rich family in $\mathcal{S}(X_1 \times X_2)$.

 \subsection{Projectional skeletons}
 Now, let us introduce the notion of \emph{projectional skeletons}, proposed in \cite{kubis2009banach}. 
 This tool has proved to be both efficient and prominent in characterizing properties of nonseparable spaces by means of seemingly simpler techniques that rely on the separability of the space itself, or on the separability of certain associated function spaces.
 
 \begin{definition}
 	A \emph{projectional skeleton} in the (rather non-separable) Banach space 
 	\((X, \| \cdot \|)\) is a family of linear bounded projections 
 	\((P_s : s \in \Gamma)\) on \(X\), indexed by a directed and 
 	\(\sigma\)-complete set \((\Gamma, \leq)\) such that
 	\begin{enumerate}
 		\item[(i)] \(P_s (X)\) is separable for every \(s \in \Gamma\),
 		\item[(ii)] \(X = \displaystyle\bigcup_{s \in \Gamma} P_s (X)\),
 		\item[(iii)] \(P_t \circ P_s = P_s = P_s \circ P_t\) whenever \(s,t \in \Gamma\) and \(s \leq t\), and
 		\item[(iv)] Given a sequence \(s_1 \leq s_2 \leq \cdots\) in \(\Gamma\), we have
 		\[
 		P_s (X) = \overline{\bigcup_{n \in \mathbb{N}} P_{s_n} (X)},
 		\]
 		where \(s := \displaystyle\sup_{n \in \mathbb{N}} s_n\).
 	\end{enumerate}
 	For $r \geq 1$, we say that $(P_s : s \in \Gamma)$ is an \emph{$r$-projectional skeleton} if it is a projectional skeleton and 
\[
    \|P_s(x)\| \leq r \|x\| \text{ for all } s \in \Gamma \text{ and } x \in X.
\]

In addition, we say that a projectional skeleton $(P_s : s \in \Gamma)$ is \emph{commutative} if 
\[
    P_s \circ P_t = P_t \circ P_s, \text{ for all }  s,t \in \Gamma.
\]
 \end{definition}
Let $A \subset X$ be a nonempty bounded set. 
We consider the pseudometric $\rho_A$ on $X^*$ by
\[
    \rho_A(x_1^*, x_2^*) := \sup \left\{ \big| \langle x_1^* - x_2^*, a \rangle \big| : a \in A \right\}, 
    \qquad x_1^*, x_2^* \in X^*.
\]

Given $\varepsilon \geq 0$, a projectional skeleton $(P_s : s \in \Gamma)$ 
in $X$ is called \emph{$A$-$\varepsilon$-shrinking} if, for every increasing sequence 
$\gamma_1 \leq \gamma_2 \leq \cdots$ in $\Gamma$ and for every $x^* \in X^*$, we have
\[
    \limsup_{j} \rho_A \big( P_{\gamma_j}^* (x^*), 
    P_{\sup \gamma_i}^* (x^*) \big) \leq \varepsilon \|x^*\|.
\]
If $\varepsilon = 0$, we simply say that the skeleton is \emph{$A$-shrinking}.

The following is a characterization of weakly compactly generated Banach spaces with projectional skeletons, as presented in \cite{fabian2018wcg}.

 \begin{theorem}\label{familia_rica_wcg}
 	For a Banach space $(X, \|\cdot\|)$, the following are equivalent:
 	\begin{itemize}
 		\item[(i)] $X$ is weakly compactly generated.
 		\item[(ii)] There exist a bounded closed symmetric convex and linearly dense set $A \subset X$, and a \textit{(commutative 1-)} projectional skeleton $(P_\gamma : \gamma \in \Gamma)$ on $(X, \|\cdot\|)$ (with $\displaystyle\bigcup_{\gamma \in \Gamma} P_\gamma^* (X^*) = X^*$), which is $A$-shrinking and satisfies that $P_\gamma(A) \subset A$ for every $\gamma \in \Gamma$.
 	\end{itemize}
 	Moreover, the index set of the skeleton $\Gamma$ may be assumed to be a rich family in $\mathcal{S}_{\rectangle{}}(X \times X^*)$ such that for every $\gamma:=V\times Y\in \Gamma$, we have that $X= V \oplus Y^{\perp}$ and $P_\gamma$ is the projection of $X$ onto $V$, i.e, with range $V$ and kernel $Y^{\perp}$, and also $\dual{P}_{\gamma}(\dual{X})= \overline{Y}^{\dual{w}}$.
 \end{theorem}
 
 \begin{proof}
 	The equivalence between items (i) and (ii) follows from \cite[Theorem 21]{fabian2018wcg}. The fact that the index set of the skeleton $\Gamma$ may be assumed to be a rich family in $\mathcal{S}_{\rectangle{}}(X \times X^*)$ with the stated properties above it is implicit in the construction given by the authors in \cite{fabian2018wcg}; see in particular \cite[Lemma 1, Lemma 2, Remark 3, and Remark 12]{fabian2018wcg}. %Nevertheless, a shorter proof of this fact can be provided as follows....
 \end{proof}
\subsection{Preliminary Results on Separable Reduction Techniques}\label{section_ResultsOn_SEPRED}

To address the case where $X$ is nonseparable, we first establish the existence of suitable rich families. 
These tools will enable us to extend Theorem~\ref{previo_a_attouch}, previously proved under the separability assumption, 
to the general setting of WCG Banach spaces. 

{\color{red}It is important to note that all the rich families constructed in this section will be rectangle rich families, that is, subsets of $\mathcal{S}_{\rectangle{}} (X\times X^\ast)$ even when certain properties concern only the primal or the dual space individually. This convention simplifies the analysis, as it allows us to work with families that simultaneously satisfy conditions in both the primal and dual spaces by considering their intersections.}

The next two lemmas are formulated in the framework of general Banach spaces. 
They show that certain properties of functions admit separable reduction.
 
The first result concerns the nonemptiness of the domain of the upper $\Gamma$-limit, in order to show that the hyphotesis $\dom(\els f_n))\neq\emptyset$ of Theorem \ref{previo_a_attouch} is true if we restrict each $f_n$ to certain separable subspaces.
{\color{red}\begin{proposition}\label{Fam_unif_propia}
   Let
$f_n\colon X \to \R\cup\{\infty\}$, $n\in\N$, be a sequence of functions such that $\dom \els f_n \neq \emptyset$, then there is a rich family $\mathcal{R} \subseteq \mathcal{S}_{\rectangle{}}(X\times X^\ast)$ such that for every $V \times Y\in \mathcal{R}$, the sequence $(f_{n}|_V)_{n\in\N}$ satisfies $\dom (\els (f_{n}|_V)) \neq \emptyset$.
\end{proposition}

\begin{proof}
    Define $\mathcal{R}':= \{V\in \mathcal{S}(X) : \dom (\els (f_{n}|_V)) \neq \emptyset\}$.   This family is clearly rich. Consequently the family
    $ \mathcal{R} := \mathcal{R}' \times \mathcal{S}(X)$ is the desired family. 
\end{proof} 
}
 The second result establishes a separable reduction principle: the values of the conjugate of a function can be recovered from the conjugates of its restrictions to suitable separable subspaces.

\begin{proposition}\label{3.1}
	Let $X$ be a Banach space and let $f\colon X \to \R \cup \{\infty\}$ be a proper lower semicontinuous convex function.  
	Then there exists a rectangle rich family $\mathcal{R} \subseteq \mathcal{S}_{\rectangle{}}(X \times X^\ast)$ such that, for every $V \times Y \in \mathcal{R}$, one has
	\begin{align}\label{EqProp3.1}
		(f|_V)^\ast(x^\ast|_{V}) = f^\ast(x^\ast), 
		\qquad \text{for all } x^\ast \in Y.
	\end{align}
\end{proposition}

\begin{proof}
First, observe that by the definition of the conjugate the following inequality holds whenever $V \subset W$ and $x^\ast \in X^\ast$:
\begin{align*}
  (f|_V)^\ast(x^\ast|_{V}) \leq (f|_W)^\ast(x^\ast|_{W}).
\end{align*}

Now, assume that there exist $\alpha > 0$, $\beta \in \R$, and $\bar{x}^{\ast} \in X^{\ast}$ such that 
\begin{equation} \label{coercive_case}
    f(x) \geq \alpha \|x\|^2 + \langle \bar{x}^{\ast}, x \rangle + \beta, 
    \qquad \forall x \in X.
\end{equation}

From Lemma~\ref{cont_conj}, this implies that $\dual{(f|_{V})}$ is continuous for every $V \subseteq X$ such that $V \cap \dom f \neq \emptyset$.

Consider
\[
    \mathcal{R} := \bigl\{ V \times Y \in \mathcal{S}_{\rectangle{}}(X \times \dual{X}) : 
    (f|_{V})^\ast (x^\ast|_{V}) = f^\ast(x^\ast) \;\; \forall x^\ast \in Y \bigr\}.
\]

We shall prove that $\mathcal{R}$ is cofinal.  
Fix any $Z \in \mathcal{S}(X \times \dual{X})$. Since $\mathcal{S}_{\rectangle{}}(X \times \dual{X})$ is rich, there exists $V_{0} \times Y_{0} \in \mathcal{S}_{\rectangle{}}(X \times \dual{X})$ such that $Z \subseteq V_{0} \times Y_{0}$.  
Let $C_0$ be a countable dense subset of $Y_0$, and let $S_0 \subseteq X$ be a countable set such that, for every $\dual{x} \in C_{0}$,
\[
    f^\ast(\dual{x}) = \sup_{x \in S_0} \bigl\{ \langle \dual{x}, x \rangle - f(x) \bigr\}.
\]
Define 
\[
    V := \overline{\operatorname{sp}}(V_0 \cup S_0), 
    \qquad Y := Y_0.
\]

Clearly, $V \times Y \in \mathcal{S}_{\rectangle{}}(X \times \dual{X})$.  
We now show that $V \times Y \in \mathcal{R}$. Fix $\dual{x} \in Y$, and let $(\dual{x}_{n}) \subset C_0$ be a sequence converging to $\dual{x}$. By construction, for all $n \in \N$ we have
\[
    (f|_{V})^\ast (\dual{x}_n|_{V}) = f^\ast(\dual{x}_n).
\]
Using the continuity of $f^\ast$ and $(f|_{V})^\ast$, we conclude that
\[
    (f|_{V})^\ast (\dual{x}|_{V}) = f^\ast(\dual{x}).
\]

Thus $\mathcal{R}$ is nonempty and cofinal.  

Next, we prove that $\mathcal{R}$ is $\sigma$-complete. Let $(V_{n} \times Y_n)_{n \in \N}$ be an increasing sequence in $\mathcal{R}$. We must verify that
\[
    \overline{\bigcup_{n \in \N} V_{n} \times Y_n} \in \mathcal{R}.
\]
Since $(V_{n} \times Y_n)_{n\in\N}$ is increasing, it follows that
\[
    \overline{\bigcup_{n \in \N} V_{n} \times Y_n}= \overline{\bigcup_{n\in\N}V_n}\times\overline{\bigcup_{n\in\N}Y_n}.
\]
 Denote $V:= \overline{\bigcup_{n\in\N}V_n}$ and $Y:=\overline{\bigcup_{n\in\N}Y_n}$. As $\mathcal{S}_{\rectangle{}}(X \times \dual{X})$ is rich, we have $V \times Y \in \mathcal{S}_{\rectangle{}}(X \times \dual{X})$.  
Now, let $\dual{x} \in Y$ and choose a sequence $(\dual{x}_n)_{n \in \N} \subset \bigcup_{n \in \N} Y_n$ converging to $\dual{x}$. Then there exists a subsequence $(Y_{n_k})_{k\in\N}$ such that 
\[
  \dual{x}_{k} \in Y_{n_k} \text{ and } (f|_V)^\ast(\dual{x}_{k}|_{V}) 
   = (f|_{V_{n_k}})^\ast(\dual{x}_{k}|_{V_{n_k}}) 
   = f^\ast (\dual{x}_{k}), 
   \qquad \text{ for all } k \in \N.
\]
By continuity of the involved functions, we obtain
\[
    (f|_V)^\ast(\dual{x}|_{V}) = f^\ast(\dual{x}),
\]
which shows that $V \times Y \in \mathcal{R}$. Therefore, $\mathcal{R}$ is rich. 

In the general case where \eqref{coercive_case} does not necessarily hold, define for each $\lambda > 0$ the function
\[
    f^{\lambda}(x) := f(x) + \tfrac{\lambda}{2} \|x\|^{2}, 
\]
which satisfies \eqref{coercive_case}.  It follows that, for every $\lambda > 0$, the family
\[
    \mathcal{R}_\lambda := \bigl\{ V \times Y \in \mathcal{S}_{\rectangle{}}(X \times \dual{X}) : 
    (f^\lambda|_{V})^\ast (x^\ast|_{V}) = (f^\lambda)^\ast(x^\ast) \;\; \forall x^\ast \in Y \bigr\}
\]
is rich. Using the known identity $(f^\lambda|_{V})^\ast = ((f|_{V})^\ast)_{\lambda}$ (the Moreau envelope in $V^\ast$, we obtain
\[
    \mathcal{R}_\lambda = \bigl\{ V \times Y \in \mathcal{S}_{\rectangle{}}(X \times \dual{X}) : 
    ((f|_{V})^\ast)_{\lambda}(x^\ast|_{V}) = (f^\ast)_{\lambda}(x^\ast) \;\; \forall x^\ast \in Y \bigr\}.
\]

Finally, set
\[
    \mathcal{R} := \bigcap_{n \in \N} \mathcal{R}_{\frac{1}{n}}.
\]
This family is rich. Since $(\dual{(f|_{V})})_{\lambda} \nearrow \dual{(f|_{V})}$ pointwise as $\lambda \to 0$ (see Lemma~\ref{ConvMoreau}), it follows that $\mathcal{R}$ is the desired family.

\end{proof}
The next proposition establishes that, when the space $X$ is weakly compactly generated,  
the equality given in \eqref{EqProp3.1} extends to the weak$^\ast$-closure of the space $Y$.  
Formally, we state the following result.

\begin{proposition} \label{familia_rica_2}
	Let $X$ be a WCG Banach space and let $f\colon X \to \R \cup \{\infty\}$ be a  proper lower semicontinuous convex function.  
	Then there exists a rectangle rich family $\mathcal{R} \subseteq \mathcal{S}_{\rectangle{}}(X \times X^\ast)$ such that, for every $V \times Y \in \mathcal{R}$, one has
	\begin{align}\label{igualdad_conjugada_restringida}
		(f|_{V})^\ast(x^\ast|_{V}) = f^\ast(x^\ast),
		\qquad \text{for all } x^\ast \in \overline{Y}^{\dual{w}}.
	\end{align}
\end{proposition}

\begin{proof}
First, assume that $f$ is continuous at the origin.

Since $X$ is WCG, we may consider a weakly compact, convex, and symmetric set $A$ that is linearly dense in the norm topology. It then follows that  
\[
\overline{\bigcup_{n\in\N}[0,n]A} = X.
\]

Because $f$ is continuous at the origin, it is also continuous at a point of $\dom(\delta_{[0,n]A})$. Define  
\[
g_n := f + \delta_{[0,n]A},
\]  
which is a sequence of   proper lower semicontinuous convex functions. Let $\mathcal{R}_n$ be the rich family associated with $g_n$ as in Proposition~\ref{3.1}, and define the rich family  
\[
\mathcal{R}' := \bigcap_{n\in\N}\mathcal{R}_n.
\]  
We will prove that for every $V \times Y \in \mathcal{R}'$, the equality \eqref{igualdad_conjugada_restringida} holds.

Indeed, let $V \times Y \in \mathcal{R}$, and let $\dual{x} \in \overline{Y}^{\dual{w}}$. By the Mackey--Arens theorem (see, e.g., \cite[Lemma 19]{fabian2018wcg}), there exists a net $(\dual{x}_\alpha)\subseteq Y$ converging to $\dual{x}$ with respect to the Mackey topology, that is, $\dual{x}_\alpha \overset{\mu^\ast}{\to} \dual{x}$. Moreover, from \cite[Proposition 4.1.20]{Libro_Correa}, we have  
\[
\dual{g_n} = \dual{f}\, \square \,\sigma_{[0,n]A}.
\]  
Observe that  
\[
g_{n}|_{V} = f|_{V} + (\delta_{[0,n]A})|_{V} = f|_{V} + \delta_{V \cap [0,n]A},
\]  
so that  
\[
(g_n|_{V})^\ast = (f|_{V})^{\ast}\, \square \,\sigma_{V \cap [0,n]A}.
\]
Furthermore, it is easy to prove that for all $\dual{y}, \dual{z} \in \dual{X}$, and all $n \in \N$ we have  
\begin{equation}\label{lipschitz}
    |\dual{g}_n(\dual{y})- \dual{g}_n(\dual{z})|\leq\sigma_{[0,n]A}(\dual{y}-\dual{z}).
\end{equation}
Similarly,  
\begin{equation}\label{lipschitz_2}
    |\dual{(g_{n}|_{V})}(\dual{y}|_{V})- \dual{(g_{n}|_{V})}(\dual{z}|_{V})|
    \leq\sigma_{V\cap[0,n]A}(\dual{y}|_{V}-\dual{z}|_{V})
    \leq \sigma_{[0,n]A}(\dual{y}-\dual{z}).
\end{equation}

Since $\dual{x}_\alpha \overset{\mu^\ast}{\to} \dual{x}$ and $[0,n]A$ is weakly compact, convex, and symmetric, it follows from \eqref{lipschitz} and \eqref{lipschitz_2} that  
\[
\dual{(g_{n}|_{V})}(\dual{x_\alpha}) \to \dual{(g_{n}|_{V})}(\dual{x}) \quad\text{and}\quad
\dual{g}_n(\dual{x_\alpha})\to \dual{g}_n(\dual{x}).
\]

On the other hand, by the property of the rich family (see Proposition~\ref{3.1}),  
\[
\dual{(g_{n}|_{V})}(\dual{x_\alpha})=\dual{g}_n(\dual{x_\alpha}).
\]  
Hence, we conclude that  
\begin{equation}\label{igualdad_aproximaciones}
\dual{(g_n|_{V})}(\dual{x})= \dual{g}_n(\dual{x}),\text{ for all } n \in \N.
\end{equation}

Since $f$ is continuous at the origin, $f|_{V}$ is also continuous. Taking the limit in \eqref{igualdad_aproximaciones} as $n \to \infty$, and applying Lemma~\ref{aproximacion}, we obtain \eqref{igualdad_conjugada_restringida}.

Now, for the general case, let $\hat{f}$ denote the Moreau envelope of $f$, i.e., $\hat{f} := f_{\lambda}$ for some $\lambda > 0$. Then $\hat{f}$ is continuous. Let $\mathcal{R}_\lambda$ be the rich family associated with $\hat{f}$ as constructed above, and let $\mathcal{A}$ be the rich family associated with the function $h(x) := \tfrac{1}{2}\|x\|^2$. Define  
\[
\mathcal{R} := \mathcal{R}_\lambda \cap \mathcal{A}.
\]  

For every $V \times Y \in \mathcal{R}$, we have  
\begin{equation*}  
    (f|_{V})^\ast (x^\ast|_V) + \frac{\lambda}{2} \|x^\ast|_V\|_{V^\ast}^2 
    = f^\ast(x^\ast) + \frac{\lambda}{2} \|x^\ast\|^2, 
    \quad \forall x^\ast \in \overline{Y}^{w^\ast},
\end{equation*}
where $\|x^\ast|_V\|_{V^\ast}$ denotes the dual norm on $V^\ast$ induced by $\|\cdot\|$ on $V$. By the properties of $\mathcal{A}$, it follows that  
\[
\|x^\ast|_V\|_{V^\ast} = \|x^\ast\|_{\ast}.
\]
Thus, \eqref{igualdad_conjugada_restringida} holds.  

\end{proof}
Despite the fact that the previous result will be primarily used in the proof of Theorem~\ref{previo_a_attouch}, it is worth noting that it also leads to interesting insights into the structure of WCG Banach spaces. In particular, it provides a simple isometry between separable subspaces and the weak$^\ast$-closures of separable subspaces in the dual. This result is in the same spirit as \cite[Theorem~2.3]{rich_families_asplund}, but formulated for WCG spaces rather than Asplund spaces. Formally, we obtain the following result.

 { \color{red}
\begin{corollary}\label{surjective_isometry}
Let  $X$ be a WCG Banach space. Then there exists a rectangle rich family $\mathcal{R} \subseteq  \mathcal{S}_{\rectangle{}}(X\times X^\ast)$ such that for every $V\times Y \in \mathcal{R}$ and every $\dual{x}\in \dual{V}$, there exists  $\dual{y}\in\overline{Y}^{\dual{w}}$ satisfying $$\dual{y}|_{V} = \dual{x} \text{ and } \| y^\ast\|_{X^\ast}  = \| x^\ast \|_{\dual{V}}.$$
		In other words, for every $V\times Y\in\mathcal{R}$: $\overline{Y}^{\dual{w}} \ni \dual{y}\to \dual{y}|_{V}\in\dual{V}$ is a surjective isometry.
\end{corollary}
}
\begin{proof}
 {\color{red}
Let $\mathcal{R}_1$ be the rich family obtained by applying 
Proposition~\ref{familia_rica_2} to the function $\delta_\mathbb{B}$, 
where $\mathbb{B}:=\{x \in X : \norm{x} \leq 1\}$. 
Let $\mathcal{R}_2$ be the rich family from Theorem~\ref{familia_rica_wcg}, 
and recall that for $\gamma := V \times Y \in \mathcal{R}_2$ we have 
$\dual{P}_\gamma(\dual X) = \overline{Y}^{\dual{w}}$.  
Now set $\mathcal{R} := \mathcal{R}_1 \cap \mathcal{R}_2$ and consider 
$\gamma := V \times Y \in \mathcal{R}$.

Furthermore, let us introduce the following notation to view the projection 
$P_\gamma$ as a mapping from $X$ to $V$:  
\begin{equation}\label{NOtation2}
    \tilde{P}_\gamma \colon X \to V, 
    \qquad \tilde{P}_\gamma(x) := P_\gamma(x).
\end{equation}
In this terminology, the adjoint operator of $\tilde{P}_\gamma$ is the mapping 
$\tilde{P}_\gamma^\ast \colon V^\ast \to X^\ast$. 
With this notation, we obtain
\begin{equation*}
    \dual{\tilde{P}}_\gamma(\dual{V}) 
    = \dual{P}_\gamma(\dual{X}).
\end{equation*}

Therefore, for every $x^\ast \in V^\ast$ there exists 
$y^\ast := \tilde{P}_\gamma^\ast(x^\ast)$ such that 
$\dual{y}|_{V} = \dual{x}$. 
This shows that the mapping  
\[
    \overline{Y}^{\dual{w}} \ni \dual{y} \;\longmapsto\; 
    \dual{y}|_{V} \in \dual{V}
\]
is surjective. 

Finally, since $\mathcal{R} \subseteq \mathcal{R}_1$, we have for every 
$\dual{y} \in \overline{Y}^{\dual{w}}$ that  
\[
    \norm{\dual{y}}_{\dual{X}} 
    = \norm{\dual{y}|_{V}}_{\dual{V}}.
\]  
Hence, the mapping is an isometry, which completes the proof.

}
\end{proof}

The next proposition shows that the equicoercivity property \eqref{equicoercividad} is separably reducible. 
{\color{red} \begin{proposition}\label{Rich_Family_First:Theo}
    Let $X$ be a WCG Banach space. Let $f_n\colon X\to \R\cup\{\infty\}$, $n\in\N$, be a sequence of proper lower semicontinuous convex functions satisfying  the equicoercivity property \eqref{equicoercividad}. Then there exists a rectangle rich family $\mathcal{R} \subseteq  \mathcal{S}_{\rectangle{}} (X\times X^\ast)$ such that for every $V\times Y \in \mathcal{R}$ and every sequence $ (\dual{v}_n)_{n\in\N} \subseteq V^\ast$ 
\begin{align*}
   \sup_{ n \in \mathbb{N}} (f_{n}|_{V})^\ast (\dual{v}_n) <\infty  \Rightarrow (\dual{v}_n)_{n\in\N} \text{ is bounded }.
\end{align*}
    
\end{proposition}

\begin{proof}
   For each $n\in\N$, let $\mathcal{R}_n$ be the rich family from Proposition~\ref{familia_rica_2} associated with $f_n$, and let $\mathcal{F}$ be the rich family from Corollary~\ref{surjective_isometry}. Define  
\begin{align*}
    \mathcal{R}:= \mathcal{F} \cap \bigcap_{n\in\N}\mathcal{R}_n,
\end{align*}  
which is itself a rich family by Proposition~\ref{rich families}.   Let $\gamma := V\times Y \in \mathcal{R}$ and $(\dual{v}_n) \subseteq V^\ast$ with  
\[
\sup_{n\in\N} (f_{n}|_{V})^\ast(\dual{v}_n) < \infty.
\]  

On one hand, by the property of $\mathcal{F}$, there exists $y^\ast_n \in Y$ such that $y^\ast_n|_{V} = v^\ast_n$ and $\| v^\ast_n \| = \| y^\ast_n \|$. On the other hand, by the properties of each $\mathcal{R}_n$, we have  
\[
\dual{f_n}(\dual{y}_n) = \dual{(f_{n}|_{V})}(\dual{y}_n|_{V}) = \dual{(f_{n}|_{V})}(\dual{v}_n).
\]  

Therefore,  
\[
\sup_{n\in\N} f_n^\ast(y^\ast_n) < \infty.
\]  
By the equicoercivity property \eqref{equicoercividad}, the sequence $(y_n^\ast)_{n\in\N}$ is bounded, and consequently $(v_n^\ast)_{n\in\N}$ is also bounded. This concludes the proof.
\end{proof}

}
The final result of this section provides an intermediate step between Theorem~\ref{previo_a_attouch} and its generalization.  
In particular,  
it shows that there is a sufficiently large family of separable subspaces where the conclusion of Theorem~\ref{previo_a_attouch} holds.

      \begin{proposition} \label{Attouch-In_rich_f}
      let $X$ be a WCG Banach space, and let $f_n\colon X\to \R\cup\{\infty\}$ be a sequence of 
      proper lower semicontinuous convex functions such that the sequence $\dual{f}_n$ satisfies \eqref{equicoercividad} and such that $\dom(\els f_n)\neq\emptyset$. Then there is a rectangle rich family $\mathcal{R} \subseteq  \mathcal{S}_{\rectangle{}}(X\times X^\ast)$ such that for every $V\times Y \in \mathcal{R}$ the following holds
    \begin{equation*}
    	(\eli(f_{n}|_{V}))^{*}(\dual{v}) = \wels (f_{n}|_{V})^{*}(\dual{v}), \text{ for all } v^\ast \in V^\ast. 
    	\end{equation*}  
\end{proposition} 

\begin{proof}
    Let $\mathcal{A}_1$ be the rich family obtained in Proposition~\ref{Rich_Family_First:Theo}, and let $\mathcal{A}_2$ be the rich family from Proposition~\ref{Fam_unif_propia}. Set  
    \begin{equation*}
        \mathcal{R} := \mathcal{A}_{1}\cap \mathcal{A}_2.
    \end{equation*}
For every $V\times Y \in \mathcal{R}$, the hypotheses of Theorem~\ref{previo_a_attouch} are satisfied on $V$. Since we have already established the result for separable spaces, Theorem~\ref{previo_a_attouch} also holds for $V$.
\end{proof}

\subsection{Proof of Theorem \ref{previo_a_attouch} in the general case}\label{ProofMain_WCG}
We now have all the necessary ingredients to extend Theorem~\ref{previo_a_attouch} from the separable case to the general setting of WCG Banach spaces. 

\begin{proof}
Let $\mathcal{R}\subseteq \mathcal{S}_{\rectangle{}}(X\times X^\ast)$ be a rectangle rich family that simultaneously satisfies 
Theorem~\ref{familia_rica_wcg}, Proposition~\ref{familia_rica_2}, 
Corollary~\ref{surjective_isometry}, Proposition~\ref{Rich_Family_First:Theo}, 
and Proposition~\ref{Attouch-In_rich_f}.  

Denote $f := \eli f_n$ and fix $x^\ast \in X^\ast$. 
Let $\gamma := V\times Y \in \mathcal{R}$ be such that $x^\ast \in Y$. 
By Lemma~\ref{Lemma:inq_01conjugate}, we may assume that $f^\ast(x^\ast) < \infty$.  

We first observe that  
\[
  f|_{V} \;\leq\;  \eli f_{n}|_{V}.
\]
Consequently, 
\[
  f^\ast(x^\ast) 
  \;\geq\; (f|_{V})^\ast (x^\ast|_{V}) 
  \;\geq\; \bigl(\eli (f_{n}|_{V})\bigr)^{*}(x^\ast|_{V})
  \;=\; \wels (f_{n}|_{V})^{*}(x^\ast|_{V}),
\]
where the last equality follows from the properties of the rich family $\mathcal{R}$.  

Now fix $\varepsilon > 0$. By the definition of $\wels (f_{n}|_{V})^{*}$, 
there exists a sequence $(z_n^\ast) \subset V^\ast$ with 
$z_n^\ast \overset{\ast}{\rightharpoonup} x^\ast|_{V}$ such that  
\[
  \wels (f_{n}|_{V})^{*}(x^\ast|_{V}) 
  > \limsup_{n} (f_{n}|_{V})^{*}(z_n^\ast) - \varepsilon.
\]  

By Proposition~\ref{familia_rica_2} and Corollary~\ref{surjective_isometry}, 
there exist $x_n^\ast \in \overline{Y}^{\dual{w}}$ such that $x_n^\ast|_{V} = z_n^\ast$ and  
\[
  (f_{n}|_{V})^{*}(z_n^\ast) = f_{n}^{*}(x_n^\ast).
\]
Moreover, $\tilde{P}^\ast_\gamma(z_n^\ast) = x_n^\ast$ and 
$\tilde{P}^\ast_\gamma(x^\ast|_{V}) = x^\ast$, where $\tilde{P}_\gamma$ is defined in~\eqref{NOtation2}.  

We now show that $x_n^\ast \overset{\ast}{\rightharpoonup} x^\ast$. 
Indeed, for any $x \in X$,  
\begin{align*}
  \bigl|\langle x_n^\ast - x^\ast, x \rangle\bigr| 
  &= \bigl|\langle \dual{\tilde P}_{\gamma}(z_n^\ast) - \dual{\tilde P}_{\gamma}(x^\ast|_{V}), x \rangle\bigr| \\[0.5ex]
  &= \bigl|\langle z_n^\ast - x^\ast|_{V}, \tilde{P}_\gamma(x) \rangle\bigr|.
\end{align*}
Since $\tilde{P}_\gamma(x) = P_\gamma(x) \in V$ and 
$z_n^\ast \overset{\ast}{\rightharpoonup} x^\ast|_{V}$ in $V^\ast$, 
it follows that $\langle x_n^\ast - x^\ast, x \rangle \to 0$.  

Hence $x_n^\ast \overset{\ast}{\rightharpoonup} x^\ast$. Therefore,  
\[
  \limsup_{n} f_{n}^{*}(x_n^\ast) \;\geq\; \wels f_n^*(x^\ast).
\]  
Consequently,  
\[
  f^\ast(x^\ast) \;\geq\; \wels f_n^*(x^\ast) - \varepsilon.
\]  
Since $\varepsilon > 0$ was arbitrary, we deduce that  
\[
  f^\ast(x^\ast) \;\geq\; \wels f_n^*(x^\ast).
\]  
On the other hand, we always have 
$f^\ast(x^\ast) \leq \wels f_n^*(x^\ast)$ 
(see Lemma~\ref{Lemma:inq_01conjugate}). 
Combining both inequalities completes the proof.  

\end{proof}

\section{Proof of Corollary \ref{dualidad}}\label{Section:Proof_dualidad}
 
 \begin{proof}
Assume that $\welim \dual{f}_n = \dual{f}$. Then
\begin{equation*}
	\weli \dual{f}_n \;\geq\; \dual{f} \;\geq\; \wels \dual{f}_n.
\end{equation*}
Taking conjugates and using that $\dom(\wels \dual{f}_n) = \dom(\dual{f}) \neq \emptyset$, 
we can apply Theorem~\ref{teorema_2} to obtain
\begin{equation*}
	\els f_n \;\leq\; f \;\leq\; (\wels \dual{f}_n)^\ast.
\end{equation*}
Since it is always true that $\wels \dual{f}_n \geq (\eli f_n)^\ast$ 
(see Lemma~\ref{Lemma:inq_01conjugate}) and also $(\eli f_n)^{\ast\ast} \leq \eli f_n$, 
we conclude that
\begin{equation*}
	\els f_n \;\leq\; f \;\leq\; \eli f_n.
\end{equation*}
Therefore,
\begin{equation*}
	\dual{f} = \welim \dual{f}_n 
	\quad \Longrightarrow \quad 
	f = \elim f_n.
\end{equation*}

\medskip

Now we prove the reverse implication under the assumption that $X$ is a WCG Banach space and  $(\dual{f}_n)$ satisfies 
\eqref{equicoercividad}.  
Suppose that $f = \elim f_n$. Then
\begin{equation*}
	\els f_n \;\leq\; f \;\leq\; \eli f_n.
\end{equation*}
Taking conjugates and applying Theorem~\ref{previo_a_attouch} 
(since $\dom(\els f_n) = \dom(f) \neq \emptyset$), 
together with Theorem~\ref{teorema_2} 
(since $\dom(\wels \dual{f}_n) = \dom(\dual{f}) \neq \emptyset$), 
we obtain
\begin{equation*}
	(\weli \dual{f}_n)^{\ast\ast} \;\geq\; \dual{f} \;\geq\; \wels \dual{f}_n.
\end{equation*}

Since $(\weli \dual{f}_n)^{\ast\ast} \leq \weli \dual{f}_n$, we deduce that
\[
	\weli \dual{f}_n \;\geq\; \dual{f} \;\geq\; \wels \dual{f}_n,
\]
and hence $\dual{f} = \welim \dual{f}_n$. 
This completes the proof of the theorem. 
 \end{proof}

\section{Proof of Theorem \ref{Theo:equival}} \label{Section:Proof_Theo:equival}

In this section we prove Theorem~\ref{Theo:equival}.  
Before establishing the implications $(a) \Rightarrow (b) \Rightarrow (c)$, let us first observe that the equivalence of the three items, under the additional equicoercivity assumption \eqref{equicoercividad} for the sequence $(f_n^\ast)$, follows directly from Corollary~\ref{dualidad}, which already shows the equivalence between $(a)$ and $(c)$.  

Hence, in the remainder of this section we divide the proof of the implications into two subsections.

\subsection{Proof of Theorem \ref{Theo:equival}: $(a)\Longrightarrow (b)$}

\begin{proof} Let $(x,y^\ast)\in \partial f$, that is,
\[
    f(x)+\dual{f}(y^\ast)-\inner{y^\ast}{x}=0.
\]
By hypothesis, there exists a sequence $(z_n^\ast)_{n \in \N} \subseteq \dual{X}$ such that 
$z_n^\ast \overset{\ast}{\rightharpoonup} y^\ast$ and 
$\dual{f}_n(z_n^\ast)\to\dual{f}(y^\ast)$. 
By Corollary~\ref{dualidad}, there exists a sequence $(u_n)_{n\in \N}\subseteq X$ such that 
$u_n \to x$ and $f_n(u_n)\to f(x)$.  

Define
\[
    \varepsilon_n := f_n(u_n) + \dual{f}_n(z_n^\ast) - \inner{z_n^\ast}{u_n} > 0,
\]
which is a sequence of positive numbers converging to $0$.  

By the Brøndsted–Rockafellar  theorem (see, e.g., ~\cite{BronstedRockafellar1965} or \cite[Theorem 4.3.2]{Libro_borwein}), for every $n\in\N$ one can find 
$(x_n, y_n^\ast)\in \partial f_n$ such that
\[
    \norm{z_n^\ast - y_n^\ast} \leq \sqrt{\varepsilon_n}
    \quad\text{and}\quad 
    \norm{x_n - u_n} \leq \sqrt{\varepsilon_n}.
\]

It follows that $x_n \to x$ and $y_n^\ast \overset{\ast}{\rightharpoonup} y^\ast$. 
Since $(x_n,y_n^\ast)\in \partial f_n$, we conclude that
\[
    \partial f_n \xrightarrow{G(\|\cdot\|\times \dual{w})} \partial f.
\]

\end{proof}
\subsection{Proof of Theorem \ref{Theo:equival}: $(b)\Longrightarrow (c)$}

In order to prove that $(b)\Longrightarrow (c)$ in Theorem \ref{Theo:equival}, we will recall  the following integration Lemma given in \cite[Lemma 2.3]{Thibault}. For that purpose let us introduce the following notation.  For a function $\varphi\colon X \to \mathbb{R} \cup \{-\infty, \infty\}$   (not necessarily convex), the Fenchel subdifferential of $\varphi$ at a point $x \in X$ is defined as
\[
\partial^{\text{Fen}} \varphi(x) = \{x^* \in X^* : \langle x^*, u - x \rangle + \varphi(x) \leq \varphi(u), \, \forall u \in X\},
\]
if $x \in \operatorname{dom} \varphi := \{x \in X : \varphi(x) < \infty\}$. Otherwise, if $x \notin \operatorname{dom} \varphi$, we set:
\[
\partial^{\text{Fen}} \varphi(x) = \emptyset.
\]

\begin{lemma}\label{inclusion subdiferencial}
    Let $g\colon X \to \mathbb{R} \cup \{\infty\}$ be a lower semicontinuous function and $f: X \to \mathbb{R} \cup \{\infty\}$ be a proper lower semicontinuous convex function. If:
\[
\partial f \subseteq \partial^{\text{Fen}} g,
\]
then $f = g + c$ , for some $c\in\R$.

\end{lemma}
The following technical lemma will be useful in the sequel. Since its proof is only a slight modification of \cite[Lemma~2.4]{Thibault}, we omit it here.
\begin{lemma}\label{lemma:attouch_simetric}
    Let $f_n\colon X \to \mathbb{R} \cup \{\infty\}$, $n \in \mathbb{N}$, be proper lower semicontinuous convex functions satisfying that $\eli f_n > -\infty $ and there exists a strongly convergent sequence $(a_n)_{n\in\N}$ for which 
\[
\displaystyle\limsup_n f_n(a_n) < \infty.
\]
Then, for any sequence $(x_n, x_n^*)$ with $(x_n, x_n^*) \in \partial f_n$ and $x_n \to x$, $\dual{x}_n \overset{\ast}{\rightharpoonup} \dual{x}$, one has:
\begin{enumerate}
    \item[(a)] $\eli f_n(x) = \displaystyle\liminf_n f_n(x_n)$ and $\els f_n(x) = \displaystyle\limsup_n f_n(x_n)$.
    \item[(b)] $(x, x^*) \in \partial^{\text{Fen}} (\eli f_n) \cap \partial (\els f_n)$.
\end{enumerate}
Furthermore, we have:
\[
\Li_{\norm{\cdot} \times w^\ast} \partial f_n \subseteq \partial^{\text{Fen}} (\eli f_n) \cap \partial (\els f_n).
\]
\end{lemma}

Now we prove $(b)\Longrightarrow (c)$ from Theorem \ref{Theo:equival}. \begin{proof}
Indeed, let us assume that (b) holds, i.e., 
\[
\partial f_n \xrightarrow{G(\norm{\cdot}\times w^\ast)} \partial f  
\]
and the normalization condition (N.C.) holds, in other words,  there exists $(a,a^*) \in \partial f$ and a sequence  $(a_n,a_n^*) \in \partial f_n$ such that 
    $a_n \to a,\;  
  a_n^* \overset{\ast}{\rightharpoonup} a^* \text{ and }  
    f_n(a_n) \to f(a).$

We need to show that $f= \eli f_n$ and $f= \els f_n$.

The N.C. ensures  that the assumptions of Lemma \ref{lemma:attouch_simetric} are satisfied. Hence, using  Lemma \ref{lemma:attouch_simetric}, we deduce that:
\[
\partial f \subseteq  \Li_{\norm{\cdot} \times w^\ast} \partial f_n  \subseteq \partial^{\text{Fen}}(\eli f_n) \cap \partial(\els f_n).
\]
From Lemma \ref{inclusion subdiferencial}, we obtain that:
\[
f = \eli f_n + C_1 = \els f_n + C_2,
\]
where \( C_1 \) and \( C_2 \) are constant real numbers. In addition, Lemma \ref{lemma:attouch_simetric}  the N.C. imply that:
\[
\eli f_n(a) = f(a) = \els f_n(a),
\]
for some \( a \in \operatorname{dom}(f) \). Hence, \( C_1 = C_2 = 0 \) and the proof is finished.

\end{proof}

% ---------------------------------------------------------------
% REFERENCES
% ---------------------------------------------------------------
\bibliographystyle{acm}
\bibliography{references}

@book {Libro_funcional,
    AUTHOR = {Fabian, Mari\'an and Habala, Petr and H\'ajek, Petr and
              Montesinos, Vicente and Zizler, V\'aclav},
     TITLE = {Banach space theory},
    SERIES = {CMS Books in Mathematics/Ouvrages de Math\'ematiques de la
              SMC},
      NOTE = {The basis for linear and nonlinear analysis},
 PUBLISHER = {Springer, New York},
      YEAR = {2011},
     PAGES = {xiv+820},
      ISBN = {978-1-4419-7514-0},
   MRCLASS = {46-02 (46Bxx)},
  MRNUMBER = {2766381},
MRREVIEWER = {Sophocles\ Mercourakis},
       DOI = {10.1007/978-1-4419-7515-7},
       URL = {https://doi.org/10.1007/978-1-4419-7515-7},
}

@article {MR377482,
    AUTHOR = {Odell, E. and Rosenthal, H. P.},
     TITLE = {A double-dual characterization of separable {B}anach spaces
              containing {$l\sp{1}$}},
   JOURNAL = {Israel J. Math.},
  FJOURNAL = {Israel Journal of Mathematics},
    VOLUME = {20},
      YEAR = {1975},
    NUMBER = {3-4},
     PAGES = {375--384},
      ISSN = {0021-2172},
   MRCLASS = {46B10},
  MRNUMBER = {377482},
MRREVIEWER = {Pierre\ Billard},
       DOI = {10.1007/BF02760341},
       URL = {https://doi.org/10.1007/BF02760341},
}

@book {MR737004,
    AUTHOR = {Diestel, Joseph},
     TITLE = {Sequences and series in {B}anach spaces},
    SERIES = {Graduate Texts in Mathematics},
    VOLUME = {92},
 PUBLISHER = {Springer-Verlag, New York},
      YEAR = {1984},
     PAGES = {xii+261},
      ISBN = {0-387-90859-5},
   MRCLASS = {46B15 (04A15 04A20 46-01 46B20)},
  MRNUMBER = {737004},
MRREVIEWER = {J.\ Bourgain},
       DOI = {10.1007/978-1-4612-5200-9},
       URL = {https://doi.org/10.1007/978-1-4612-5200-9},
}

@book {Libro_borwein,
    AUTHOR = {Borwein, Jonathan M. and Vanderwerff, Jon D.},
     TITLE = {Convex functions: constructions, characterizations and
              counterexamples},
    SERIES = {Encyclopedia of Mathematics and its Applications},
    VOLUME = {109},
 PUBLISHER = {Cambridge University Press, Cambridge},
      YEAR = {2010},
     PAGES = {x+521},
      ISBN = {978-0-521-85005-6},
   MRCLASS = {49-02 (46B20 46N10 49J53 52A41 90C25)},
  MRNUMBER = {2596822},
MRREVIEWER = {Heinz\ H.\ Bauschke},
       DOI = {10.1017/CBO9781139087322},
       URL = {https://doi.org/10.1017/CBO9781139087322},
}

@book {Rockafellar_Wets,
    AUTHOR = {Rockafellar, R. Tyrrell and Wets, Roger J.-B.},
     TITLE = {Variational analysis},
    SERIES = {Grundlehren der mathematischen Wissenschaften [Fundamental
              Principles of Mathematical Sciences]},
    VOLUME = {317},
 PUBLISHER = {Springer-Verlag, Berlin},
      YEAR = {1998},
     PAGES = {xiv+733},
      ISBN = {3-540-62772-3},
   MRCLASS = {49-02 (46N10 47N10 49J52 49K40 90C30)},
  MRNUMBER = {1491362},
MRREVIEWER = {Francis\ H.\ Clarke},
       DOI = {10.1007/978-3-642-02431-3},
       URL = {https://doi.org/10.1007/978-3-642-02431-3},
}

@book {Libro_Correa,
    AUTHOR = {Correa, Rafael and Hantoute, Abderrahim and L\'opez, Marco A.},
     TITLE = {Fundamentals of convex analysis and optimization---a supremum
              function approach},
    SERIES = {Springer Series in Operations Research and Financial
              Engineering},
 PUBLISHER = {Springer, Cham},
      YEAR = {[2023] \copyright 2023},
     PAGES = {xiii+444},
      ISBN = {978-3-031-29550-8; 978-3-031-29551-5},
   MRCLASS = {90-01 (49J53 90C25 90C46 90C48)},
  MRNUMBER = {4647344},
       DOI = {10.1007/978-3-031-29551-5},
       URL = {https://doi.org/10.1007/978-3-031-29551-5},
}

@book{Libro_Attouch,
  author    = {Attouch, Hedy},
  title     = {Variational Convergence for Functions and Operators},
  year      = {1984},
  publisher = {Pitman},
  address   = {Boston, MA},
  series    = {Pitman Advanced Publishing Program},
  isbn      = {978-0273085949}
}

@article{Thibault,
  author    = {Combari, Christophe and Thibault, Lionel},
  title     = {On the Graph Convergence of Subdifferentials of Convex Functions},
  journal   = {Proceedings of the American Mathematical Society},
  volume    = {126},
  number    = {8},
  pages     = {2231--2240},
  year      = {1998},
  doi       = {10.1090/S0002-9939-98-04724-8},
  url       = {https://www.ams.org/journals/proc/1998-126-08/S0002-9939-98-04724-8/}
}

@article{rich_families_asplund,
  author    = {Cúth, Marek and Fabian, Marian},
  title     = {Asplund spaces characterized by rich families and separable reduction of Fréchet subdifferentiability},
  journal   = {Journal of Functional Analysis},
  volume    = {270},
  number    = {4},
  pages     = {1361--1378},
  year      = {2016},
  doi       = {10.1016/j.jfa.2015.10.014},
  url       = {https://doi.org/10.1016/j.jfa.2015.10.014}
}

@article{penot2004continuity,
  author    = {Jean-Paul Penot and Constantin Zălinescu},
  title     = {Continuity of the Legendre--Fenchel transform for some variational convergences},
  journal   = {Optimization},
  volume    = {53},
  number    = {5--6},
  pages     = {549--562},
  year      = {2004},
  doi       = {10.1080/02331930412331329533},
  url       = {https://doi.org/10.1080/02331930412331329533}
}

@article{kubis2009banach,
	title={Banach spaces with projectional skeletons},
	author={Kubi{\'s}, Wies{\l}aw},
	journal={Journal of mathematical analysis and applications},
	volume={350},
	number={2},
	pages={758--776},
	year={2009},
	publisher={Elsevier}
}

@article{beer1992slice,
  author    = {Gerald Beer},
  title     = {The slice topology: a viable alternative to Mosco convergence in nonreflexive spaces},
  journal   = {Nonlinear Analysis: Theory, Methods \& Applications},
  volume    = {19},
  number    = {3},
  pages     = {271--290},
  year      = {1992},
  doi       = {10.1016/0362-546X(92)90145-5},
  url       = {https://doi.org/10.1016/0362-546X(92)90145-5}
}

@book{aliprantis,
  author    = {Charalambos D. Aliprantis and Kim C. Border},
  title     = {Infinite Dimensional Analysis: A Hitchhiker's Guide},
  edition   = {3},
  publisher = {Springer},
  address   = {Berlin; London},
  year      = {2006},
  isbn      = {978-3-540-29587-7},
  doi       = {10.1007/3-540-29587-9}
}

@article{fabian2018wcg,
  author    = {Marian Fabian and Vicente Montesinos Santaluc{\'u}a},
  title     = {WCG spaces and their subspaces grasped by projectional skeletons},
  journal   = {Functiones et Approximatio Commentarii Mathematici},
  volume    = {59},
  number    = {2},
  pages     = {231--250},
  year      = {2018},
  doi       = {10.7169/facm/1721},
  url       = {https://doi.org/10.7169/facm/1721}
}

@book{zalinescu_convex_analysis,
  author    = {Constantin Z\u{a}linescu},
  title     = {Convex Analysis in General Vector Spaces},
  publisher = {World Scientific},
  year      = {2002},
  series    = {Series on Optimization and Its Applications},
  volume    = {2},
  isbn      = {978-9810243641},
  doi       = {10.1142/5022},
  url       = {https://doi.org/10.1142/5022}
}

@book {MR4659162,
	AUTHOR = {Thibault, Lionel},
	TITLE = {Unilateral variational analysis in {B}anach spaces. {P}art
	{I}---general theory},
	PUBLISHER = {World Scientific Publishing Co. Pte. Ltd., Singapore},
	YEAR = {[2023] \copyright 2023},
	PAGES = {xxiv+891},
	ISBN = {978-981-125-494-9},
	MRCLASS = {49-02 (49J27 49J52 49J53)},
	MRNUMBER = {4659162},
	MRREVIEWER = {Marius\ Durea},
}

@book{gamma_convergence_for_beginners,
  author    = {Braides, Andrea},
  title     = {Gamma-convergence for Beginners},
  series    = {Oxford Lecture Series in Mathematics and Its Applications},
  volume    = {22},
  publisher = {Oxford University Press},
  year      = {2002},
  isbn      = {0198507844},
  address   = {Oxford},
}

@article{Borwein,
  author    = {Jonathan M. Borwein and Warren B. Moors},
  title     = {Separable determination of integrability and minimality of the Clarke subdifferential mapping},
  journal   = {Proceedings of the American Mathematical Society},
  year      = {2000},
  volume    = {128},
  number    = {8},
  pages     = {215--221},
  doi       = {10.1090/S0002-9939-99-05015-4},
  url       = {https://doi.org/10.1090/S0002-9939-99-05015-4}
}

@article{fabian_separable_determined,
  author    = {Mari\'{a}n Fabi{\'a}n and Alexander Ioffe},
  title     = {Separable reduction in the theory of Fr\'{e}chet subdifferentials},
  journal   = {Set-Valued and Variational Analysis},
  volume    = {21},
  number    = {4},
  pages     = {661--671},
  year      = {2013},
  doi       = {10.1007/s11228-013-0256-1},
  url       = {https://doi.org/10.1007/s11228-013-0256-1}
}

@article{rich_families_subdifferential,
author = {Fabian, Marián and Ioffe, Alexander},
year = {2016},
month = {01},
pages = {631-648},
title = {Separable reductions and rich families in the theory of Fréchet subdifferentials},
journal= {Journal of convex analysis},
volume = {23}
}

@book{annals,
  author    = {Joram Lindenstrauss and David Preiss and Jaroslav T{\v{i}}ser},
  title     = {Fr{\'e}chet Differentiability of Lipschitz Functions and Porous Sets in Banach Spaces},
  series    = {Annals of Mathematics Studies},
  volume    = {179},
  publisher = {Princeton University Press},
  year      = {2012},
  isbn      = {9780691153551},
  doi       = {10.1515/9781400842698},
  url       = {https://doi.org/10.1515/9781400842698}
}

@article{attouch1977,
  author  = {H{\'e}dy Attouch},
  title   = {Convergence de fonctions convexes, des sous-diff{\'e}rentiels et semi-groupes associ{\'e}s},
  journal = {Comptes Rendus de l’Académie des Sciences, Séries A–B},
  volume  = {284},
  number  = {10},
  pages   = {A539--A542},
  year    = {1977},
}

@article{mosco1969convergence,
  author    = {Mosco, Umberto},
  title     = {Convergence of convex sets and of solutions of variational inequalities},
  journal   = {Advances in Mathematics},
  volume    = {3},
  pages     = {510--585},
  year      = {1969},
  doi       = {10.1016/0001-8708(69)90009-7},
  url       = {https://doi.org/10.1016/0001-8708(69)90009-7}
}

@article{mosco1971continuity,
  author    = {Umberto Mosco},
  title     = {On the continuity of the Young--Fenchel transform},
  journal   = {Journal of Mathematical Analysis and Applications},
  volume    = {35},
  number    = {2},
  pages     = {518--535},
  year      = {1971},
  doi       = {10.1016/0022-247X(71)90200-9},
  url       = {https://doi.org/10.1016/0022-247X(71)90200-9}
}

@article{mosco1967approximation,
  author    = {Umberto Mosco},
  title     = {Approximation of the solutions of some variational inequalities},
  journal   = {Annali della Scuola Normale Superiore di Pisa – Scienze Fisiche e Matematiche},
  series    = {Ser.\ 3},
  volume    = {21},
  number    = {3},
  pages     = {373--394},
  year      = {1967},
  publisher = {Scuola Normale Superiore di Pisa},
  url       = {https://www.numdam.org/item/ASNSP_1967_3_21_3_373_0/}
}

@article{joly,
  author    = {Jean-Luc Joly},
  title     = {Une famille de topologies sur l'ensemble des fonctions convexes pour lesquelles la polarité est bicontinue},
  journal   = {Journal de Mathématiques Pures et Appliquées},
  series    = {9\textsuperscript{e} série},
  volume    = {52},
  number    = {10},
  pages     = {421--441},
  year      = {1973},
}

@article{moscoreflexivity,
  author    = {Gerald Beer and Jonathan M. Borwein},
  title     = {Mosco convergence and reflexivity},
  journal   = {Proceedings of the American Mathematical Society},
  volume    = {109},
  number    = {2},
  pages     = {427--436},
  year      = {1990},
}

@article{parabolic_variational_inequailities,
  author    = {H\'edy Attouch and Alain Damlamian},
  title     = {Strong solutions for parabolic variational inequalities},
  journal   = {Nonlinear Analysis: Theory, Methods \& Applications},
  volume    = {2},
  number    = {3},
  pages     = {329--353},
  year      = {1978},
  doi       = {10.1016/0362-546X(78)90021-4},
  url       = {https://doi.org/10.1016/0362-546X(78)90021-4}
}

@incollection{numerical_attouch,
  author    = {B. Lemaire},
  title     = {Coupling optimization methods and variational convergence},
  booktitle = {Trends in Mathematical Optimization},
  editor    = {Karl-Heinz Hoffmann and Jochem Zowe and Jean-Baptiste Hiriart-Urruty and Claude Lemaréchal},
  series    = {International Series of Numerical Mathematics},
  volume    = {84},
  pages     = {163--179},
  publisher = {Birkhäuser Verlag},
  address   = {Basel},
  year      = {1988},
  doi       = {10.1007/978-3-0348-9297-1_13},
  url       = {https://doi.org/10.1007/978-3-0348-9297-1_13}
}

@article{rockafellar1990generalized,
  author    = {R. T. Rockafellar},
  title     = {Generalized second derivatives of convex functions and saddle functions},
  journal   = {Transactions of the American Mathematical Society},
  volume    = {322},
  number    = {1},
  pages     = {51--77},
  year      = {1990},
  doi       = {10.1090/S0002-9947-1990-1031242-0},
  url       = {https://doi.org/10.1090/S0002-9947-1990-1031242-0}
}

@article{daniilidis2024extending,
  author    = {Aris Daniilidis and Marc Quincampoix},
  title     = {Extending the Rademacher Theorem to Set-Valued Maps},
  journal   = {SIAM Journal on Optimization},
  volume    = {34},
  number    = {2},
  pages     = {1784--1798},
  year      = {2024},
  doi       = {10.1137/22M1538831},
  url       = {https://doi.org/10.1137/22M1538831}
}

@book{attouch2014variational,
  author    = {Hedy Attouch and Giuseppe Buttazzo and Gérard Michaille},
  title     = {Variational Analysis in Sobolev and BV Spaces: Applications to PDEs and Optimization},
  edition   = {2},
  series    = {MOS–SIAM Series on Optimization},
  publisher = {Society for Industrial and Applied Mathematics},
  year      = {2014},
  isbn      = {978-1-61197-347-1},
  doi       = {10.1137/1.9781611973488}
}

@inproceedings{Mosco1968,
  author    = {Mosco, Umberto},
  title     = {Perturbation of variational inequalities},
  booktitle = {Proceedings of the Symposium on Nonlinear Functional Analysis},
  series    = {Proceedings of Symposia in Pure Mathematics},
  year      = {1968},
  address   = {Chicago, 1968},
  publisher = {American Mathematical Society},
}

@article{DeGiorgiFranzoni1975,
  author    = {Ennio De Giorgi and Tullio Franzoni},
  title     = {Su un tipo di convergenza variazionale},
  journal   = {Rendiconti della Accademia Nazionale dei Lincei. Classe di Scienze Fisiche, Matematiche e Naturali},
  series    = {8},
  volume    = {58},
  year      = {1975},
  pages     = {842--850}
}

@article{DeGiorgiSpagnolo1973,
  author  = {De Giorgi, Ennio and Spagnolo, Sergio},
  title   = {Sulla convergenza degli integrali dell’energia per operatori ellittici del secondo ordine},
  journal = {Bollettino dell'Unione Matematica Italiana},
  series  = {IV},
  volume  = {8},
  pages   = {391--411},
  year    = {1973},
  language= {Italian}
}

@article{Spagnolo1968,
  author  = {Spagnolo, Sergio},
  title   = {Sulla convergenza di soluzioni di equazioni paraboliche ed ellittiche},
  journal = {Annali della Scuola Normale Superiore di Pisa – Classe di Scienze Fisiche e Matematiche},
  series  = {Serie 3},
  volume  = {22},
  number  = {4},
  pages   = {577--597},
  year    = {1968},
  language= {Italian},
  url     = {https://www.numdam.org/item/ASNSP_1968_3_22_4_571_0/}
}

@article{Spagnolo1967cauchy,
  author  = {Spagnolo, Sergio},
  title   = {Sul limite delle soluzioni di problemi di Cauchy relativi all'equazione del calore},
  journal = {Annali della Scuola Normale Superiore di Pisa -- Classe di Scienze Fisiche e Matematiche},
  series  = {Serie 3},
  volume  = {21},
  number  = {4},
  pages   = {657--699},
  year    = {1967},
  language= {Italian},
  url     = {https://www.numdam.org/item/ASNSP_1967_3_21_4_657_0/}
}

@book{dalmaso1993introduction,
  author    = {Dal Maso, Gianni},
  title     = {An Introduction to $\Gamma$-Convergence},
  series    = {Progress in Nonlinear Differential Equations and Their Applications},
  volume    = {8},
  publisher = {Birkhäuser},
  address   = {Boston},
  year      = {1993},
  isbn      = {978-0-8176-3679-1},
  doi       = {10.1007/978-1-4612-0327-8}
}

@article{Wijsman1966convexII,
  author    = {Wijsman, R. A.},
  title     = {Convergence of sequences of convex sets, cones and functions. II},
  journal   = {Transactions of the American Mathematical Society},
  volume    = {123},
  number    = {1},
  pages     = {32--45},
  year      = {1966},
  doi       = {10.1090/S0002-9947-1966-0196599-8}
}

@article{Wijsman1964convexI,
  author    = {Wijsman, R. A.},
  title     = {Convergence of sequences of convex sets, cones and functions. I},
  journal   = {Bulletin of the American Mathematical Society},
  volume    = {70},
  number    = {2},
  pages     = {186--188},
  year      = {1964},
  doi       = {10.1090/S0002-9904-1964-11175-1}
}

@article{BronstedRockafellar1965,
  author  = {Br{\o}nsted, Arne and Rockafellar, R. Tyrrell},
  title   = {On the subdifferentiability of convex functions},
  journal = {Proceedings of the American Mathematical Society},
  volume  = {16},
  number  = {4},
  pages   = {605--611},
  year    = {1965},
  doi     = {10.2307/2034005}
}

@article{rockafellar1966extension,
  author    = {Rockafellar, R. Tyrrell},
  title     = {Extension of Fenchel's duality theorem for convex functions},
  journal   = {Duke Mathematical Journal},
  volume    = {33},
  number    = {1},
  pages     = {81--89},
  year      = {1966},
  doi       = {10.1215/S0012-7094-66-03312-6},
  url       = {https://doi.org/10.1215/S0012-7094-66-03312-6}
}

@article{BurkholderWijsman1963,
  author    = {Burkholder, D. L. and Wijsman, R. A.},
  title     = {Optimum properties and admissibility of sequential tests},
  journal   = {Annals of Mathematical Statistics},
  volume    = {34},
  number    = {1},
  pages     = {1--17},
  year      = {1963},
  month     = {mar},
  doi       = {10.1214/aoms/1177704238},
  url       = {https://projecteuclid.org/journals/annals-of-mathematical-statistics/volume-34/issue-1/Optimum-Properties-and-Admissibility-of-Sequential-Tests/10.1214/aoms/1177704238.full}
}

@article{McLindenBergstrom1981,
  author    = {L. McLinden and R. C. Bergstrom},
  title     = {Preservation of convergence of convex sets and functions in finite dimensions},
  journal   = {Transactions of the American Mathematical Society},
  volume    = {268},
  number    = {1},
  pages     = {127--142},
  year      = {1981},
  doi       = {10.1090/S0002-9947-1981-0628449-5}
}

@article{Back1986continuity,
  author  = {Back, Kerry},
  title   = {Continuity of the Fenchel Transform of Convex Functions},
  journal = {Proceedings of the American Mathematical Society},
  volume  = {97},
  pages   = {661--667},
  number  = {4},
  year    = {1986}
}

@article{Beer1988YoungFenchel,
  author    = {Beer, Gerald},
  title     = {On the Young–Fenchel transform for convex functions},
  journal   = {Proceedings of the American Mathematical Society},
  volume    = {104},
  number    = {4},
  pages     = {1115--1123},
  year      = {1988},
  doi       = {10.1090/S0002-9939-1988-0937844-8},
  url       = {https://doi.org/10.1090/S0002-9939-1988-0937844-8}
}

@article{Beer1990EpiDistance,
  author    = {Beer, Gerald},
  title     = {Conjugate convex functions and the epi-distance topology},
  journal   = {Proceedings of the American Mathematical Society},
  volume    = {108},
  pages     = {117--126},
  year      = {1990},
  number    = {1}
}

@article{AttouchWets1983SaddleConvergence,
  author    = {Attouch, Hédy and Wets, Roger J.-B.},
  title     = {A Convergence Theory for Saddle Functions},
  journal   = {Transactions of the American Mathematical Society},
  volume    = {280},
  number    = {1},
  pages     = {1--41},
  year      = {1983},
  doi       = {10.1090/S0002-9947-1983-0712247-X},
  url       = {https://doi.org/10.1090/S0002-9947-1983-0712247-X}
}

@article{DoleckiSalinettiWets1983,
  author    = {Dolecki, Szymon and Salinetti, Gabriella and Wets, Roger J.-B.},
  title     = {Convergence of functions: equi-semicontinuity},
  journal   = {Transactions of the American Mathematical Society},
  volume    = {276},
  number    = {1},
  pages     = {409--429},
  year      = {1983},
  doi       = {10.1090/S0002-9947-1983-0684518-7},
  url       = {https://doi.org/10.1090/S0002-9947-1983-0684518-7}
}

@article{AttouchWets1986Isometries,
  author    = {Attouch, H{\'e}dy and Wets, Roger J.-B.},
  title     = {Isometries for the Legendre--Fenchel Transform},
  journal   = {Transactions of the American Mathematical Society},
  volume    = {296},
  number    = {1},
  pages     = {33--60},
  year      = {1986},
  doi       = {10.1090/S0002-9947-1986-0837797-X},
  url       = {https://doi.org/10.1090/S0002-9947-1986-0837797-X}
}

@article{ContessePenot1991,
  author    = {Contesse, L. and Penot, J.-P.},
  title     = {Continuity of the Fenchel correspondence and continuity of polarities},
  journal   = {Journal of Mathematical Analysis and Applications},
  volume    = {156},
  number    = {2},
  pages     = {305--328},
  year      = {1991},
  doi       = {10.1016/0022-247X(91)90399-K},
  url       = {https://doi.org/10.1016/0022-247X(91)90399-K}
}

@article{Penot1991CosmicBoundedHausdorff,
  author    = {Penot, Jean-Paul},
  title     = {The cosmic Hausdorff topology, the bounded Hausdorff topology, and continuity of polarity},
  journal   = {Proceedings of the American Mathematical Society},
  volume    = {113},
  number    = {2},
  pages     = {275--285},
  year      = {1991},
  doi       = {10.1090/PSPAM/113/2/275},
  url       = {https://doi.org/10.1090/PSPAM/113/2/275}
}

@article{AzeAttouchWets1988,
  author    = {Az{\'e}, Dominique and Attouch, H{\'e}dy and Wets, Roger J.-B.},
  title     = {Convergence of convex-concave saddle functions: applications to convex programming and mechanics},
  journal   = {Annales de l'I.H.P. Analyse non lin{\'e}aire},
  volume    = {5},
  number    = {6},
  pages     = {537--572},
  year      = {1988},
  doi       = {10.1016/S0294-1449(16)30335-3},
  url       = {https://www.numdam.org/item/AIHPC_1988__5_6_537_0/}
}

@article{AubinWets1988Stable,
  author    = {Aubin, Jean-Pierre and Wets, Roger J.-B.},
  title     = {Stable approximations of set-valued maps},
  journal   = {Annales de l'Institut Henri Poincar{\'e}, Analyse Non Lin{\'e}aire},
  volume    = {5},
  number    = {6},
  pages     = {519--535},
  year      = {1988},
  doi       = {10.1016/S0294-1449(16)30334-1},
  url       = {https://doi.org/10.1016/S0294-1449(16)30334-1}
}

@article{AttouchWets1989EpigraphicalAnalysis,
  author    = {Attouch, H{\'e}dy and Wets, Roger J.-B.},
  title     = {Epigraphical Analysis},
  journal   = {Annales de l'Institut Henri Poincar{\'e} Analyse Non Lin{\'e}aire},
  volume    = {6},
  number    = {Suppl.},
  pages     = {73--100},
  year      = {1989},
  doi       = {10.1016/S0294-1449(17)30036-7},
  url       = {https://doi.org/10.1016/S0294-1449(17)30036-7}
}

@article{AttouchWets1991QuantitativeEpigraphical,
  author    = {Attouch, H{\'e}dy and Wets, Roger J.-B.},
  title     = {Quantitative Stability of Variational Systems: I. The Epigraphical Distance},
  journal   = {Transactions of the American Mathematical Society},
  volume    = {328},
  number    = {2},
  pages     = {695--729},
  year      = {1991},
  doi       = {10.1090/S0002-9947-1991-1018570-0},
  url       = {https://doi.org/10.1090/S0002-9947-1991-1018570-0}
}

@article{AttouchWets1993NonlinearConditioning,
  author    = {Attouch, H{\'e}dy and Wets, Roger J.-B.},
  title     = {Quantitative Stability of Variational Systems: II. A Framework for Nonlinear Conditioning},
  journal   = {SIAM Journal on Optimization},
  volume    = {3},
  number    = {2},
  pages     = {359--381},
  year      = {1993},
  doi       = {10.1137/0803016},
  url       = {https://doi.org/10.1137/0803016}
}

@article{AttouchWets1993III,
  author       = {Hédy Attouch and Roger J.-B. Wets},
  title        = {Quantitative stability of variational systems: III. $\varepsilon$‑approximate solutions},
  journal      = {Mathematical Programming},
  volume       = {61},
  number       = {1‑3},
  pages        = {197--214},
  year         = {1993},
  doi          = {10.1007/bf01582147}
}

@article{Zagrodny2005WeakStarConvergenceSubdifferentials,
  author    = {Zagrodny, Dariusz},
  title     = {On the Weak\* Convergence of Subdifferentials of Convex Functions},
  journal   = {Journal of Convex Analysis},
  volume    = {12},
  number    = {1},
  pages     = {213--219},
  year      = {2005}
}

@article{BeerThera1994AWDifferentialOperator,
  author    = {Beer, Gerald and Théra, Michel},
  title     = {Attouch-Wets convergence and a differential operator for convex functions},
  journal   = {Proceedings of the American Mathematical Society},
  volume    = {122},
  number    = {3},
  pages     = {851--858},
  year      = {1994},
  doi       = {10.1090/PROC/8400},
  url       = {https://doi.org/10.1090/PROC/8400}
}

@article{SonntagZalinescu1993SetConvergences,
  author    = {Sonntag, Yves and Z{\u{a}}linescu, Constantin},
  title     = {Set convergences. An attempt of classification},
  journal   = {Transactions of the American Mathematical Society},
  volume    = {340},
  number    = {1},
  pages     = {199--226},
  year      = {1993},
  doi       = {10.1090/S0002-9947-1993-1173857-8},
  url       = {https://doi.org/10.1090/S0002-9947-1993-1173857-8}
}

@article{BeerLucchetti1993WeakTopologiesClosedSubsets,
  author    = {Beer, Gerald and Lucchetti, Roberto},
  title     = {Weak topologies for the closed subsets of a metrizable space},
  journal   = {Proceedings of the American Mathematical Society},
  volume    = {335},
  number    = {2},
  pages     = {805--822},
  year      = {1993},
  doi       = {10.1090/S0002-9939-1993-1094552-X},
  url       = {https://doi.org/10.1090/S0002-9939-1993-1094552-X}
}

@article{BorweinFitzpatrick1989,
  author    = {Jonathan M. Borwein and Simon Fitzpatrick},
  title     = {Mosco convergence and the Kadec property},
  journal   = {Proceedings of the American Mathematical Society},
  volume    = {106},
  number    = {3},
  pages     = {843--851},
  year      = {1989},
  doi       = {10.2307/2047444},
  url       = {https://doi.org/10.2307/2047444}
}

@article{BeerLucchetti1991ConvexOptimizationEpiDistance,
  author    = {Beer, Gerald and Lucchetti, Roberto},
  title     = {Convex optimization and the epi-distance topology},
  journal   = {Proceedings of the American Mathematical Society},
  volume    = {327},
  number    = {2},
  pages     = {795--813},
  year      = {1991},
  doi       = {10.1090/S0002-9947-1991-1012526-X},
  url       = {https://doi.org/10.1090/S0002-9947-1991-1012526-X}
}

@article{LevyPoliquinThibault1995PartialExtensions,
  author    = {Levy, A.\ B. and Poliquin, René A. and Thibault, Lionel},
  title     = {Partial Extensions of Attouch's Theorem with Applications to Proto-Derivatives of Subgradient Mappings},
  journal   = {Transactions of the American Mathematical Society},
  volume    = {347},
  number    = {4},
  pages     = {1269--1294},
  year      = {1995},
  doi       = {10.1090/S0002-9947-1995-1290725-3},
  url       = {https://doi.org/10.1090/S0002-9947-1995-1290725-3}
}

@article{BorweinVanderwerff1995LipschitzRegularizations,
  author    = {Borwein, Jonathan M. and Vanderwerff, Jon D.},
  title     = {Convergence of Lipschitz regularizations of convex functions},
  journal   = {Journal of Functional Analysis},
  volume    = {128},
  number    = {1},
  pages     = {139--162},
  year      = {1995},
  doi       = {10.1006/jfan.1995.1048},
  url       = {https://doi.org/10.1006/jfan.1995.1048}
}

@article{BeerDiConcilio1991UniformContinuity,
  author    = {Beer, Gerald and Di Concilio, Anna},
  title     = {Uniform Continuity on Bounded Sets and the Attouch-Wets Topology},
  journal   = {Proceedings of the American Mathematical Society},
  volume    = {112},
  number    = {1},
  pages     = {235--235},
  year      = {1991},
  doi       = {10.1090/S0002-9939-1991-1033956-1},
  url       = {https://doi.org/10.1090/S0002-9939-1991-1033956-1}
}

@article{Beer1994WijsmanRenorming,
  author    = {Beer, Gerald},
  title     = {Wijsman convergence of convex sets under renorming},
  journal   = {Nonlinear Analysis: Theory, Methods \& Applications},
  volume    = {22},
  number    = {3},
  pages     = {207--216},
  year      = {1994},
  doi       = {10.1016/0362-546X(94)90034-5},
  url       = {https://doi.org/10.1016/0362-546X(94)90034-5}
}

@article{Penot1992TopologiesConvergences,
  author       = {Penot, Jean-Paul},
  title        = {Topologies and Convergences on the Space of Convex Functions},
  journal      = {Nonlinear Analysis: Theory, Methods \& Applications},
  volume       = {18},
  number       = {10},
  pages        = {905--916},
  year         = {1992},
  doi          = {10.1016/0362-546X(92)90128-2},
}

@article{Rockafellar1988EpiDifferentiability,
  author       = {Rockafellar, R. T.},
  title        = {First- and second-order epi-differentiability in nonlinear programming},
  journal      = {Transactions of the American Mathematical Society},
  volume       = {307},
  number       = {1},
  pages        = {75--108},
  year         = {1988},
  doi          = {10.1090/S0002-9947-1988-0936806-9},
  mrnumber     = {MR936806}
}

@article{Cominetti1991PseudoDifferentiability,
  author       = {Roberto Cominetti},
  title        = {On pseudo‐differentiability},
  journal      = {Transactions of the American Mathematical Society},
  volume       = {324},
  number       = {2},
  pages        = {843--865},
  year         = {1991},
  doi          = {10.1090/S0002-9947-1991-0992605-3},
}

@article{CominettiCorrea1990GenSecondOrder,
  author    = {Roberto Cominetti and Rafael Correa},
  title     = {A Generalized Second‑Order Derivative in Nonsmooth Optimization},
  journal   = {SIAM Journal on Control and Optimization},
  volume    = {28},
  number    = {4},
  pages     = {789--809},
  year      = {1990},
  doi       = {10.1137/0328045},
}

@article{Poliquin1992,
  author    = {René A. Poliquin},
  title     = {An extension of Attouch’s theorem and its application to second‐order epi‐differentiation of convexly composite functions},
  journal   = {Transactions of the American Mathematical Society},
  volume    = {332},
  number    = {2},
  pages     = {861--874},
  year      = {1992},
  doi       = {10.1090/S0002-9947-1992-1145732-5},
  publisher = {American Mathematical Society}
}

@article{LionsStampacchia1967,
  author    = {Lions, J.-L. and Stampacchia, G.},
  title     = {Variational inequalities},
  journal   = {Communications on Pure and Applied Mathematics},
  volume    = {20},
  number    = {3},
  pages     = {493--519},
  year      = {1967},
  doi       = {10.1002/cpa.3160200302},
  url       = {https://doi.org/10.1002/cpa.3160200302}
}

@article{Marcellini1973ConvergenzaFunzioniConvesse,
  author    = {Marcellini, Paolo},
  title     = {Su una convergenza di funzioni convesse},
  journal   = {Bollettino dell'Unione Matematica Italiana},
  series    = {Serie IV},
  volume    = {8},
  pages     = {137--158},
  year      = {1973}
}

@article{BoccardoMarcellini1976VariationalInequalitiesConvergence,
  author    = {Boccardo, Lucio and Marcellini, Paolo},
  title     = {Sulla convergenza delle soluzioni di disequazioni variazionali},
  journal   = {Annali di Matematica Pura ed Applicata},
  series    = {IV},
  volume    = {110},
  pages     = {137--159},
  year      = {1976}
}

@article{MarcelliniSbordone1977DualitaPerturbazione,
  author    = {Marcellini, Paolo and Sbordone, Carlo},
  title     = {Dualità e perturbazione di funzionali integrali},
  journal   = {Ricerche di Matematica},
  volume    = {26},
  pages     = {383--421},
  year      = {1977}
}

@book{AubinFrankowska1990,
  author    = {Aubin, Jean-Pierre and Frankowska, H{\'e}l{\`e}ne},
  title     = {Set-Valued Analysis},
  publisher = {Birkhäuser Boston},
  address   = {Boston, MA, USA},
  series    = {Systems \& Control: Foundations \& Applications},
  volume    = {2},
  year      = {1990},
  isbn      = {9780817638169},
  doi       = {10.1007/978-0-8176-4848-0},
  url       = {https://doi.org/10.1007/978-0-8176-4848-0}
}

@article{FitzpatrickLewis2006WeakStarConvexSets,
  author    = {Fitzpatrick, Simon and Lewis, Adrian S.},
  title     = {Weak* Convergence of Convex Sets},
  journal   = {Journal of Convex Analysis},
  volume    = {13},
  number    = {3-4},
  pages     = {711--719},
  year      = {2006}
}

@article{Rockafellar1968IntegralsConvexFunctionals,
  author  = {Rockafellar, R. Tyrrell},
  title   = {Integrals which are convex functionals},
  journal = {Pacific Journal of Mathematics},
  volume  = {24},
  number  = {3},
  pages   = {525--539},
  year    = {1968},
  doi     = {10.2140/pjm.1968.24.525},
  url     = {https://doi.org/10.2140/pjm.1968.24.525}
}

@article{Rockafellar1969MeasurableDependence,
  author  = {Rockafellar, R. Tyrrell},
  title   = {Measurable Dependence of Convex Sets and Functions on Parameters},
  journal = {Journal of Mathematical Analysis and Applications},
  volume  = {28},
  number  = {1},
  pages   = {4--25},
  year    = {1969},
  doi     = {10.1016/0022-247X(69)90104-8},
  url     = {https://doi.org/10.1016/0022-247X(69)90104-8}
}

@article{Diestel1975L1WCG,
  author  = {Diestel, Joseph},
  title   = {{$L^1(\mu;X)$ is weakly compactly generated if $X$ is}},
  journal = {Proceedings of the American Mathematical Society},
  volume  = {48},
  number  = {2},
  pages   = {508--510},
  year    = {1975},
  doi     = {10.1090/S0002-9939-1975-0367651-X},
  url     = {https://doi.org/10.1090/S0002-9939-1975-0367651-X}
}

@book{Rockafellar1974ConjugateDuality,
  author    = {Rockafellar, R. Tyrrell},
  title     = {Conjugate Duality and Optimization},
  publisher = {Society for Industrial and Applied Mathematics (SIAM)},
  address   = {Philadelphia, PA, USA},
  series    = {CBMS-NSF Regional Conference Series in Applied Mathematics},
  volume    = {16},
  year      = {1974},
  isbn      = {978-0-89871-165-2}
}

@book{CastaingValadier1977,
  author    = {Castaing, Charles and Valadier, Michel},
  title     = {Convex Analysis and Measurable Multifunctions},
  publisher = {Springer},
  address   = {Berlin, Heidelberg},
  series    = {Lecture Notes in Mathematics},
  volume    = {580},
  year      = {1977},
  doi       = {10.1007/BFb0087685},
  isbn      = {978-3-540-08138-8}
}

@article{AmbrosettiSbordone1976GammaG,
  author    = {Ambrosetti, Antonio and Sbordone, Carlo},
  title     = {$\Gamma$-convergenza e $G$-convergenza per problemi non lineari di tipo ellittico},
  journal   = {Bollettino dell'Unione Matematica Italiana},
  series    = {Serie V},
  volume    = {13-A},
  pages     = {352--362},
  year      = {1976}
}

\end{document}